\documentclass[
final, nomarks
]{dmtcs-episciences}

\usepackage[utf8]{inputenc}
\usepackage{subfigure}

\usepackage{amsmath,amsthm,amssymb}     
\usepackage{hyperref} 
\usepackage{tikz}
\usetikzlibrary{patterns.meta}

\newtheorem{theorem}{Theorem}

\theoremstyle{definition}

\newtheorem{prop}{Proposition}
\newtheorem{lemma}{Lemma}
\newtheorem*{game}{Game Rules}

\author{Lara Pudwell\affiliationmark{1}}

\title[On an Erd\H{o}s-Szekeres Game]{On an Erd\H{o}s-Szekeres Game}

\affiliation{Valparaiso University, Valparaiso, IN, USA}


\begin{document}
\publicationdata{vol. 27:1, Permutation Patterns 2024}{2026}{6}{10.46298/dmtcs.14855}{2024-11-28; 2024-11-28; 2025-11-27}{2026-02-27}
\maketitle
\begin{abstract}
\vspace{1\baselineskip}
We consider a two-player permutation game inspired by the celebrated Erd\H{o}s-Szekeres Theorem. The game depends on two positive integer parameters $a$ and $b$ and we determine the winner and give a winning strategy when $a \geq b$ and $b \in \left\{2,3,4,5\right\}$.
\end{abstract}
\section{Introduction}\label{S:intro}

Let $\mathcal{S}_n$ be the set of all permutations of $\{1,2,\dots, n\}$.  We say that permutation $\pi \in \mathcal{S}_n$ \emph{contains} $\rho \in \mathcal{S}_m$ if $\pi$ contains a subsequence $\pi_{i_1}\pi_{i_2}\cdots \pi_{i_m}$ such that $i_1<i_2<\cdots < i_m$ and $\pi_{i_a} < \pi_{i_b}$ if and only if $\rho_a < \rho_b$.  Otherwise we say $\pi$ \emph{avoids} $\rho$. One of the oldest theorems that can rephrased in terms of permutation patterns is the Erd\H{o}s-Szekeres Theorem \cite{ES35}, which was first published in 1935, and is phrased in terms of patterns as Theorem \ref{T:ES} below.

\begin{theorem}
\label{T:ES}

Any permutation of length $n \geq (a-1)(b-1)+1$ contains either an increasing subsequence of length $a$ or a decreasing subsequence of length $b$.
\end{theorem}

Since we are primarily interested in the monotone patterns, we refer to $12\cdots a$ as $I_a$ and $b \cdots 1$ as $J_b$.

While there are a variety of proofs of Theorem \ref{T:ES}, one of the most concise was given by Seidenberg \cite{S59} in 1959, using an application of the pigeonhole principle, which we revisit here:

\begin{proof}[of Theorem \ref{T:ES}]
Let $\pi  = \pi_1 \cdots \pi_n \in \mathcal{S}_n$.  For each $\pi_i$ associate an ordered pair $(a_i,d_i)$ where $a_i$ is the length of the longest increasing subsequence of $\pi$ ending in $\pi_i$ and $d_i$ is the length of the longest decreasing subsequence of $\pi$ ending in $\pi_i$.  Clearly, for all $i$, $a_i \geq 1$ and $d_i \geq 1$ since the entry $\pi_i$ is itself an increasing (resp. decreasing) subsequence of length 1.

However, we also have that if $i \neq j$, then $(a_i, d_i) \neq (a_j, d_j)$.  This follows from the fact that the entries of $\pi$ are distinct members of $\left\{1,\dots, n\right\}$.  Without loss of generality, suppose $i<j$.  If $\pi_i <\pi_j$, then $a_i <a_j$ since appending $\pi_j$ onto the increasing subsequence of length $a_i$ ending at $\pi_i$ produces an increasing subsequence of length $a_i+1$ ending at $\pi_j$.  Similarly, if $\pi_i>\pi_j$, then $d_i < d_j$.

We have $n$ distinct ordered pairs of positive integers associated with $\pi$.  If $a_i \geq a$ or $d_i \geq b$ for some $i$ then $\pi$ contains an $I_a$ or a $J_b$ pattern.  So, if $\pi$ avoids $I_a$ and $J_b$, then $1 \leq a_i \leq a-1$ and $1 \leq d_i \leq b-1$ for all $i$, which means $n \leq (a-1)(b-1)$.  Taking the contrapositive, if $n \geq (a-1)(b-1)+1$, then $\pi$ contains either $I_a$ or $J_b$ as a pattern.
\end{proof}

In 1983, Harary, Sagan, and West \cite{HSW83} studied a game based on Theorem \ref{T:ES} with rules as follows: Consider the set of integers $\left\{1,2,\dots, (a-1)(b-1)+1\right\}$.  Two players take turns selecting numbers from this set until either an increasing subsequence of length $a$ or a decreasing subsequence of length $b$ is formed. In normal play, the first player to complete such a subsequence wins.  In the mis\`{e}re version of the game, the first player to complete such a subsequence loses. The analysis in \cite{HSW83} is computer-aided, and is limited by the computer memory available at the time.  In particular, each state of the game can be labeled as winning or losing for player 1 based on an analysis of subsequent possible moves.  They determined the winning player for games where $(a-1)(b-1)+1 \leq 15$, and since the tree of game states grows exponentially in $a$ and $b$, they predicted that it is prohibitive to push computer analysis much further.  

In 2009, Albert et. al. \cite{AAAHHMS09} considered the game of Harary, Sagan, and West as well as a number of generalizations.  In particular, they also considered the play of the game on $\mathbb{Q}$ rather than on a finite set.  In this latter setting they determined that the winner of the game on $\mathbb{Q}$ with parameters $a$ and $b$ is the same as the winner of the mis\`{e}re game with parameters $a-1$ and $b-1$ (Proposition 7 of \cite{AAAHHMS09}).  They also determined the winner of the game for $a \geq b$ and $b \leq 5$ (Theorem 9 of \cite{AAAHHMS09}).  The cases where $b=2$ and $b=3$ are simple, but they gave an explicit strategy for a first player win in the cases where $b=4$ and $b=5$.  In addition, for arbitrarily large $b$, they determine that the game where $a=b$ is a first player win for $b \geq 4$.

A collaborative version of this game has also been used as a teaching tool to build intuition about Theorem \ref{T:ES} (see \cite{P21}). Of note, the current author \cite{P24} determined how many maximum length permutations avoid $I_a$ and $J_b$ for the $a \geq b=3$ case via a bijection that tracks the positions and the values of the left-to-right maxima of the relevant permutations. 

In this paper, we consider the two-player competitive version of this game, which is a reformulation of the game on $\mathbb{Q}$ studied in \cite{AAAHHMS09}. In particular, after the first $n$ moves, the current game state is a permutation of length $n$, and the next player can play any entry in $\left\{1,\dots,n+1\right\}$ as their move, appending it to the current permutation.  We will primarily analyze the mis\`{e}re game, where the first player who completes an $I_a$ pattern or a $J_b$ pattern loses.  We give strategies that work for specific choices of $b$ but for any $a \geq b$.  In contrast to the formulation in \cite{AAAHHMS09}, we model the game as shading cells in a 2-dimensional grid while they model it by a building a ternary word with certain restrictions.  For completeness, we show how to use the grid-shading model to see that the parity of $a$ determines the winner when $b=2$ and to give a strategy for a first player win in the avoidance game for $b=3$ and $b=4$.  Since \cite{AAAHHMS09} focuses primarily on normal play, these strategies match their results for $b=3$, $b=4$, and $b=5$.  In addition, we use our grid-shading model to give a winning strategy for the $b=5$ case of mis\`{e}re play, which goes beyond the scope of winning strategies given in \cite{AAAHHMS09}.

The organization of the rest of this paper is as follows.  In Section \ref{S:rep} we develop a visual way of representing game moves, inspired by Seidenberg's proof of Theorem \ref{T:ES}, which was given above.  In Sections \ref{S:2}, \ref{S:3}, \ref{S:4}, and \ref{S:5} we give a general winning strategy for the cases of mis\`{e}re play where $a \geq b$ and $2 \leq b \leq 5$.  Finally, in Section \ref{S:other}, we conclude with observations and topics for future investigation.

\section{Representation of Moves}\label{S:rep}

As described in Section \ref{S:intro}, we consider a two-player game where players take turns appending a new entry onto a permutation.  On the $n$th turn of the game, a player may play any number in $\left\{1,\dots,n\right\}$.  After the $n$th turn of the game, the current game state is a permutation $\pi$ of length $n$.  For all $i \geq 1$, the pattern formed by the first $i$ entries remains unchanged as the game progresses, but we need to track increasing and decreasing subsequences within the permutation being constructed. We call a game that ends when a player completes an $I_a$ or a $J_b$ pattern an $(a,b)$-\emph{game}.  The Seidenberg proof of Theorem \ref{T:ES} tracks increasing and decreasing subsequences using ordered pairs of positive integers.  Motivated by this representation we define the \emph{board} of an $(a,b)$-game to be a grid with $a-1$ columns and $b-1$ rows of cells.  Each cell is indexed by the ordered pair $(c,r)$ where $c$ denotes the column number and $r$ denotes the row number of the cell.

Now, for the $n$th move of an $(a,b)$-game, we shade a cell $(c,r)$ if the longest increasing subsequence of $\pi$ ending in $\pi_n$ has length $c$ and the longest decreasing subsequence of $\pi$ ending in $\pi_n$ has length $r$.  Figure \ref{F:example_board} shows the board in the $(6,5)$-game corresponding to $\pi=163425$. (Note that the hatched regions in the figure are \emph{not} shaded and will be introduced later.)  To check, $\pi_1=1$ corresponds to $(1,1)$, $\pi_2=6$ corresponds to $(2,1)$, $\pi_3=3$ corresponds to $(2,2)$, $\pi_4=4$ corresponds to $(3,2)$, $\pi_5=2$ corresponds to $(2,3)$, and $\pi_6=5$ corresponds to $(4,2)$. Note that the increasing subsequence length is given first for consistency of notation with the name of $(a,b)$-game, while these values correspond to column numbers (rather than row numbers) in the board for vertically efficient use of the page throughout this manuscript.

At the $n$th move, we know that for each $j<n$ either $\pi_j<\pi_n$ (in which case the column number of the $n$th shaded cell is larger than the column number of the $j$th shaded cell) or $\pi_j>\pi_n$ (in which case the row number of the $n$th shaded cell is larger than the row number of the $j$th shaded cell).  To this end, in addition to the shaded cells we call a cell  $(c^*,r^*)$ \emph{eliminated} if $(c,r)$ is shaded and both $c^* \leq c$ and $r^* \leq r$. The cells $(3,1)$, $(4,1)$, $(1,2)$, and $(1,3)$ are eliminated in Figure \ref{F:example_board} and thus marked with hatching.  These are cells that are ineligible to become shaded in future turns, based on the permutation formed so far in the game.  The shaded cells are a subset of the eliminated cells at any point of the game.

\begin{figure}
\begin{center}
\begin{tikzpicture}
\draw (0,0) -- (0,4)--(5,4)--(5,0)--(0,0);
\fill[gray] (0,3) rectangle (2,4);
\fill[gray] (1,2) rectangle (4,3);
\fill[gray] (1,1) rectangle (2,2);
\fill[pattern={Hatch}, pattern color=gray] (2,3) rectangle (4,4);
\fill[pattern={Hatch}, pattern color=gray] (0,1) rectangle (1,3);
\draw[step=1.0,black,thin] (0,0) grid (5,4);
\node at (0.5,3.5) {(1,1)};
\node at (1.5,3.5) {(2,1)};
\node at (2.5,3.5) {(3,1)};
\node at (3.5,3.5) {(4,1)};
\node at (4.5,3.5) {(5,1)};
\node at (0.5,2.5) {(1,2)};
\node at (1.5,2.5) {(2,2)};
\node at (2.5,2.5) {(3,2)};
\node at (3.5,2.5) {(4,2)};
\node at (4.5,2.5) {(5,2)};
\node at (0.5,1.5) {(1,3)};
\node at (1.5,1.5) {(2,3)};
\node at (2.5,1.5) {(3,3)};
\node at (3.5,1.5) {(4,3)};
\node at (4.5,1.5) {(5,3)};
\node at (0.5,0.5) {(1,4)};
\node at (1.5,0.5) {(2,4)};
\node at (2.5,0.5) {(3,4)};
\node at (3.5,0.5) {(4,4)};
\node at (4.5,0.5) {(5,4)};
\end{tikzpicture}
\end{center}
\caption{The board corresponding to $\pi=163425$ in a $(6,5)$-game}
\label{F:example_board}
\end{figure}

We now make some observations about the eliminated region of a board at any point in an $(a,b)$-game.  

First, by definition of eliminated cells, at any point in the game, the set of eliminated board cells forms one contiguous region.  Moreover, this region tells us clearly which cells are available to be shaded by the next move of the game.  We refer to cells that are not eliminated as \emph{open} cells.

\begin{prop}\label{P:legalmoves}
Suppose $S$ is the set of eliminated cells after the $n$th turn of the $(a,b)$-game.  Then move $n+1$ corresponds to shading an open cell that is edge-adjacent to a cell is $S$.  Moreover, every open cell that is edge-adjacent to $S$ corresponds to a possible next move.
\end{prop}

\begin{proof}
We consider the possible values of $(c,r)$ corresponding to $\pi_{n+1}$.  Let $i$ be the maximum column number of a cell in $S$ and let $d$ be the maximum row number of a cell in $S$.

First, consider the extreme cases where $\pi_{n+1}=1$ or $\pi_{n+1}=n+1$.  If $\pi_{n+1}=1$, then playing $\pi_{n+1}$ corresponds to cell $(1,d+1)$.  This is edge adjacent to $S$ since some cell in $S$ has row number $d$, and therefore cell $(1,d)$ is eliminated.  Similarly, if $\pi_{n+1} = n+1$, playing $\pi_{n+1}$ corresponds to cell $(i+1,1)$.  This is also edge adjacent to $S$ since some cell has column number $i$, and therefore cell $(i,1)$ is eliminated.

Now suppose that $\pi_{n+1}=j$ ($1<j<n+1$) corresponds to shading the cell $(c,r)$.  We consider what cell would be shaded if $\pi_{n+1}={j+1}$ instead in four cases. In all cases, the entries $\pi_1\cdots \pi_n$ have the same relative order.  The only difference is that if $\pi_{n+1}=j$, there exists $1 \leq k \leq n$ such that $\pi_k = j+1$, while if $\pi_{n+1}=j+1$, instead $\pi_k=j$.  When $\pi_k=j+1$ and $\pi_{n+1}=j$, it is not possible for $\pi_k$ to be part of the increasing subsequence of length $c$ ending in $\pi_{n+1}$ since $\pi_k>\pi_{n+1}$.
\begin{itemize}
\item $\pi_k$ played no role in the decreasing subsequence of length $r$ ending in $\pi_{n+1}=j$, and so $\pi_{n+1}=j$ and $\pi_{n+1}=j+1$ result in shading the same $(c,r)$ cell.
\item $\pi_k$ corresponds to shading cell $(c,r-1)$ and $\pi_k$ was part of the decreasing subsequence of length $r$ when $\pi_{n+1}=j$.  In the situation where $\pi_{n+1}=j+1$, $\pi_k \pi_{n+1}$ are the last two entries in an increasing subsequence of length $c+1$, while $\pi_{n+1}$ ends a decreasing subsequence of length $r-1$.  In other words, $\pi_{n+1}=j+1$ corresponds to shading $(c+1,r-1)$.
\item $\pi_k$ corresponds to shading cell $(c^*,r-1)$ where $c^*<c$ and $\pi_k=j+1$ played a role in forming the decreasing subsequence of length $r$ when $\pi_{n+1}=j$.  Now that $\pi_{n+1}={j+1}$, $\pi_{n+1}$ completes a decreasing subsequence of length $r-1$, while $c$ remains unchanged.  In other words, $\pi_{n+1}=j+1$ corresponds to shading $(c,r-1)$.
\item $\pi_k$ corresponds to shading cell $(c,r^*)$ where $r^*<r-1$.  Since $r^*<r-1$, there must be a different subsequence that contributed to the decreasing pattern of length $r$ when $\pi_{n+1}=j$, and so when $\pi_{n+1}=j+1$, the decreasing subsequence ending in $\pi_{n+1}$ remains the same.  However, the increasing subsequence length goes up by 1.  In other words, $\pi_{n+1}=j+1$ corresponds to shading cell $(c+1,r)$.
\end{itemize}

These four cases describe the set of possible cells that can be shaded by various choices of $\pi_{n+1}$.  We gave specific examples that showed $(1,d+1)$ and $(i+1,1)$ are possible.  We also saw that each open cell that could be shaded by a choice of $\pi_{n+1}$ differs in row and/or column number by at most 1 from another open cell.  In fact, the only time when both numbers change is when they are immediately below and immediately right of a eliminated cell that forms a southeast corner of $S$.  This uniquely describes the open cells that are edge adjacent to $S$.
\end{proof}

As an example, consider Figure \ref{F:nextmoves}, which shows the possible ordered pairs corresponding to next moves in a $(6,5)$-game whose current permutation is $\pi=163425$.  The left side of the figure shows the the points $(i,\pi_i)$ for $1 \leq i \leq 6$ with ordered pairs given for various choices of $\pi_{7}$, while the right side of the figure shows the eliminated cells after the sixth turn, and highlights the cells corresponding to choices of $\pi_7$.

\begin{figure}
\begin{center}
\scalebox{0.8}{\begin{tikzpicture}
\draw (0,0) -- (0,6)--(6,6)--(6,0)--(0,0);
\draw[step=1.0,black,thin] (0,0) grid (6,6);
\fill[black] (0.5,0.5) circle (0.3cm); 
\fill[black] (1.5,5.5) circle (0.3cm); 
\fill[black] (2.5,2.5) circle (0.3cm); 
\fill[black] (3.5,3.5) circle (0.3cm); 
\fill[black] (4.5,1.5) circle (0.3cm); 
\fill[black] (5.5,4.5) circle (0.3cm); 
\fill[red!50] (6,5.5) rectangle (7,6.5);
\node at (6.5,6) {$(5,1)$};
\fill[orange!50] (6,4.5) rectangle (7,5.5);
\node at (6.5,5) {$(5,2)$};
\fill[yellow!50] (6,3.5) rectangle (7,4.5);
\node at (6.5,4) {$(4,3)$};
\fill[green!50] (6,2.5) rectangle (7,3.5);
\node at (6.5,3) {$(3,3)$};
\fill[green!50] (6,1.5) rectangle (7,2.5);
\node at (6.5,2) {$(3,3)$};
\fill[blue!50] (6,0.5) rectangle (7,1.5);
\node at (6.5,1) {$(2,4)$};
\fill[violet!50] (6,-0.5) rectangle (7,0.5);
\node at (6.5,0) {$(1,4)$};
\end{tikzpicture}}\hspace{0.1in}\raisebox{0.5in}{\scalebox{0.8}{\begin{tikzpicture}
\draw (0,0) -- (0,4)--(5,4)--(5,0)--(0,0);
\fill[gray] (0,3) rectangle (2,4);
\fill[gray] (1,2) rectangle (4,3);
\fill[gray] (1,1) rectangle (2,2);
\fill[pattern={Hatch}, pattern color=gray] (2,3) rectangle (4,4);
\fill[pattern={Hatch}, pattern color=gray] (0,1) rectangle (1,3);
\fill[red!50] (4,3) rectangle (5,4);
\fill[orange!50] (4,2) rectangle (5,3);
\fill[yellow!50] (3,1) rectangle (4,2);
\fill[green!50] (2,1) rectangle (3,2);
\fill[blue!50] (1,0) rectangle (2,1);
\fill[violet!50] (0,0) rectangle (1,1);
\draw[step=1.0,black,thin] (0,0) grid (5,4);
\node at (0.5,3.5) {(1,1)};
\node at (1.5,3.5) {(2,1)};
\node at (2.5,3.5) {(3,1)};
\node at (3.5,3.5) {(4,1)};
\node at (4.5,3.5) {(5,1)};
\node at (0.5,2.5) {(1,2)};
\node at (1.5,2.5) {(2,2)};
\node at (2.5,2.5) {(3,2)};
\node at (3.5,2.5) {(4,2)};
\node at (4.5,2.5) {(5,2)};
\node at (0.5,1.5) {(1,3)};
\node at (1.5,1.5) {(2,3)};
\node at (2.5,1.5) {(3,3)};
\node at (3.5,1.5) {(4,3)};
\node at (4.5,1.5) {(5,3)};
\node at (0.5,0.5) {(1,4)};
\node at (1.5,0.5) {(2,4)};
\node at (2.5,0.5) {(3,4)};
\node at (3.5,0.5) {(4,4)};
\node at (4.5,0.5) {(5,4)};
\end{tikzpicture}}}
\end{center}
\caption{Ordered pairs representing possible next moves in a $(6,5)$-game where $\pi=163425$.}
\label{F:nextmoves}
\end{figure}

Finally, the player who shades the $(a-1,b-1)$ cell is the winner since this cell being shaded means that the entire board is eliminated, and the next player must add a entry to the permutation that either completes an $I_a$ pattern or a $J_b$ pattern.  Now that we have made this translation from entries to boards, we may play the Erd\H{o}s-Szekeres game as a game of shading cells on a board, rather than thinking merely in terms of one-line permutation notation.

In summary, we can rephrase the Erd\H{o}s-Szekeres game as follows:

\begin{game}

Consider a $(b-1)\times (a-1)$ array of cells.  Players 1 and 2 take turns as follows:

\begin{itemize}
\item Player 1 begins by shading the cell in the top left corner.
\item For each subsequent move, a player shades a cell that is edge-adjacent to the eliminated region.  All cells that are above and/or left of their chosen cell should also be eliminated.
\item Players alternate taking turns until the board is full.  The player who claims the bottom right corner wins.
\end{itemize}
\end{game}

We remark that this game is effectively a variant of the classic combinatorial game Chomp.

Notice that there are many permutations that may correspond to eliminating the same region of a game board.  As a small example, both $\pi=132$ and $\pi=312$ correspond to a board where a $2\times 2$ region has been eliminated.  However, these are both permutations where the most recent entry corresponds to the label $(2,2)$, and so the same amount of progress has been made towards forming an $I_a$ or $J_b$ pattern.  In terms of tracking a win or loss in the permutation game, no information has been lost.

Finally, we remark that this region of eliminated cells encodes exactly the same information about $\pi$ as the ternary words used in \cite{AAAHHMS09}.  The notation in Albert et. al. is motivated directly by Schensted's bumping algorithm in \cite{S61} rather than Seidenberg's proof.  They represent game states as words on the symbols R, B, and P with at least one P, not starting with B, not ending with R, and not containing RB as a factor.  Notice that the boundary of the eliminated region in our representation is contiguous.  If we consider eliminated cells with bottom and/or right edges on the boundary of the eliminated region, we can traverse these cells from lower left to upper right, and record R if only the bottom edge is on the boundary of the eliminated region, B if only the right edge is on the boundary of the eliminated region, and P if both the bottom and right edges are on the boundary of the shaded region.  For example, the gray  eliminated region corresponding to $\pi=163425$ shown in Figure \ref{F:example_board} corresponds to the word $RPRPB$.  The word notation of Albert et. al. has the advantage that it is language theoretic.  The grid presentation in this paper has the advantage that one can play the game focused on visual geometric information.

In the following sections, we articulate a strategy for winning this permutation game in terms of board shading.  The strategies for $2 \leq b \leq 4$ exactly match the strategies that were presented in language theoretic terms in \cite{AAAHHMS09} while the strategy for $b=5$ is completely new and extends their results. We will show that the empty board is in class $\mathcal{N}$ (i.e., a next-player winning position) when $3 \leq b \leq 5$.  When $b=2$, whether the empty board is in class $\mathcal{N}$ or class $\mathcal{P}$ (i.e. a next-player losing position) depends on the parity of $a$.

\section{Strategy when b=2}\label{S:2}

Suppose $a \geq b=2$.  Using the representation described in Section \ref{S:rep},  the game board can be visualized as a one-row, $(a-1)$-column grid, and each player colors the left-most unclaimed cell on their turn.  Once the cell in column $(a-1)$ is claimed (by player 2 if $a$ is odd, or by player 1 if $a$ is even), the other player has no more legal moves and loses.

In terms of permutations, each player will play a new largest entry on their turn since playing any smaller entry creates a $J_2$ pattern and automatically loses the game.  The resulting game permutation is $\pi = I_a$, and the game ends on the $a$th turn, leading to a loss for player 1 if $a$ is odd and a loss for player 2 if $a$ is even.

\section{Strategy when b=3}\label{S:3}

We now consider the $(a,3)$-game ($a \geq 3$), and we give a winning strategy both in terms of board shading and in terms of permutation entries.

\begin{theorem}
Player 1 has a winning strategy in the $(a,3)$-game where $a \geq 3$.
\end{theorem}

\begin{proof}
The game board in this situation is a $2 \times (a-1)$ grid. After each of their first $a-2$ moves, player 1 can produce a board where the first $i$ cells of row 1 are shaded and the first $i-1$ cells of row 2 are shaded, as illustrated in Figure \ref{F:play3}.  To start, player 1 must shade 1 cell in row 1 and 0 cells in row 2.  After that, there are only two possible moves:
\begin{itemize}
\item If player 2 shades the leftmost open cell in row 2, player 1 shades the leftmost open cell in row 1.
\item If player 2 shades the leftmost open cell in row 1, player 1 shades the leftmost open cell in row 2.
\end{itemize}

For the end game, after $(a-2)$ moves for player 1 and $(a-3)$ moves for player 2, all cells except for $(a-1,1)$, $(a-2,2)$, and $(a-1,2)$ are shaded.

Because of Proposition \ref{P:legalmoves}, player 2 must shade either $(a-1,1)$ or $(a-2,2)$.  In either case, player 1 can shade $(a-1,2)$ on their next move and force a player 2 loss.
\end{proof}

In this straightforward situation, shading a cell in row 1 corresponds to playing a new left-to-right maximum, while shading a cell in row 2 corresponds to playing a non-left-to-right-maximum, in keeping with observations of maximum length permutations avoiding $I_a$ and $321$, made in \cite{P24}.  In terms of permutations, the winning strategy is as follows: Let $m_1 < m_2 < \cdots < m_{\ell}=n$ be the values of the left-to-right maxima of $\pi$ at the start of player 1's turn.

\begin{itemize}
\item If player 2's move was not $m_{\ell}$, play $\pi_{n+1}=n+1$.
\item If player 2's move was $m_{\ell}$, play $m_{\ell-1}$.
\end{itemize}

Finally, when $\pi$ has length $2a-4$, regardless of what player 2's most recent move was, player 1 plays $\pi_{2a-3} = 2a-4$, corresponding to cell $(a-1,2)$.  On their next move, player 2 will complete either an $I_a$ pattern (by playing $\pi_{2a-2} \geq 2a-3$) or $J_3$ pattern (by playing $\pi_{2a-2} \leq 2a-4$) to lose the game.

\begin{figure}
\begin{center}
\begin{tikzpicture}
\draw (0,0) -- (0,2)--(5,2)--(5,0)--(0,0);
\draw (-1,0)--(0,0);
\draw (-1,2)--(0,2);
\draw[step=1.0,black,thin] (0,0) grid (5,2);
\fill[gray] (-1,1) rectangle (4,2);
\fill[gray] (-1,0) rectangle (3,1);
\end{tikzpicture}
\end{center}
\caption{A class $\mathcal{P}$ position for the $(a,3)$-game.}
\label{F:play3}
\end{figure}

\section{Strategy when b=4}\label{S:4}

We now consider the $(a,4)$-game ($a \geq 4$).  This is the first case where the game board has sufficiently many rows that the same shading can be obtained by more than one permutation, and so we only describe it in terms of shaded boards.  However, an interested reader -- thinking in terms of permutations -- can consider the ordered pairs corresponding to each entry in the game so far to determine a next entry that follows the strategy given here.

\begin{theorem}
Player 1 has a winning strategy in the $(a,4)$-game where $a \geq 4$.
\end{theorem}

\begin{proof}
The game board in this situation is a $3 \times (a-1)$ grid. 

To begin the game, player 1 shades cell $(1,1)$.  Regardless of whether player 2 shades cell $(1,2)$ or cell $(2,1)$, player 1 responds by shading cell $(2,2)$.

After this opening sequence, player 1 can always end their turn with a board of one of the following two forms:
\begin{enumerate}
\item Rows 1 and 2 have $k$ eliminated cells and row 3 has $k-2$ eliminated cells for some $k \geq 2$.

\begin{center}
\scalebox{0.6}{
\begin{tikzpicture}
\fill[pattern={Hatch}, pattern color=gray!40] (0.5,0) rectangle (1,3);
\draw[dashed] (0.5,0)--(1,0);
\draw[dashed] (0.5,3)--(1,3);
\draw (3,3)--(6,3)--(6,0)--(1,0);
\fill[gray] (1,1) rectangle (3,3);
\end{tikzpicture}}
\end{center}

\item Row 1 has $k$ eliminated cells and rows 2 and 3 have $k-1$ eliminated cells for some $k \geq 2$.

\begin{center}
\scalebox{0.6}{\begin{tikzpicture}
\fill[pattern={Hatch}, pattern color=gray!40] (0.5,0) rectangle (1,3);
\draw[dashed] (0.5,0)--(1,0);
\draw[dashed] (0.5,3)--(1,3);
\draw (2,3)--(6,3)--(6,0)--(1,0);
\fill[gray] (1,2) rectangle (2,3);
\end{tikzpicture}}
\end{center}
\end{enumerate}

Note that the opening sequence produces a board that fits the first case.  Now we consider each available move for player 2 and how player 1 may respond.

In the first case, player 2 has four available moves: (i) $(k+1,1)$, (ii) $(k+1,2)$, (iii) $(k-1,3)$, or (iv) $(k,3)$, shown below in black.

\begin{center}
\begin{tabular}{cccc}
(i)&(ii)&(iii)&(iv)\\
\scalebox{0.5}{
\begin{tikzpicture}
\fill[pattern={Hatch}, pattern color=gray!40] (0.5,0) rectangle (1,3);
\draw[dashed] (0.5,0)--(1,0);
\draw[dashed] (0.5,3)--(1,3);
\draw (3,3)--(6,3)--(6,0)--(1,0);
\fill[gray] (1,1) rectangle (3,3);
\fill[black] (3,2) rectangle (4,3);
\end{tikzpicture}}&\scalebox{0.5}{
\begin{tikzpicture}
\fill[pattern={Hatch}, pattern color=gray!40] (0.5,0) rectangle (1,3);
\draw[dashed] (0.5,0)--(1,0);
\draw[dashed] (0.5,3)--(1,3);
\draw (3,3)--(6,3)--(6,0)--(1,0);
\fill[gray] (1,1) rectangle (3,3);
\fill[pattern={Hatch}, pattern color=black] (3,2) rectangle (4,3);
\fill[black] (3,1) rectangle (4,2);
\end{tikzpicture}}&\scalebox{0.5}{
\begin{tikzpicture}
\fill[pattern={Hatch}, pattern color=gray!40] (0.5,0) rectangle (1,3);
\draw[dashed] (0.5,0)--(1,0);
\draw[dashed] (0.5,3)--(1,3);
\draw (3,3)--(6,3)--(6,0)--(1,0);
\fill[gray] (1,1) rectangle (3,3);
\fill[black] (1,0) rectangle (2,1);
\end{tikzpicture}}&\scalebox{0.5}{
\begin{tikzpicture}
\fill[pattern={Hatch}, pattern color=gray!40] (0.5,0) rectangle (1,3);
\draw[dashed] (0.5,0)--(1,0);
\draw[dashed] (0.5,3)--(1,3);
\draw (3,3)--(6,3)--(6,0)--(1,0);
\fill[gray] (1,1) rectangle (3,3);
\fill[pattern={Hatch}, pattern color=black] (1,0) rectangle (2,1);
\fill[black] (2,0) rectangle (3,1);
\end{tikzpicture}}
\end{tabular}
\end{center}

In each of these cases, player 1 has a natural response, shown below in blue.  

\begin{center}
\begin{tabular}{cccc}
(i)&(ii)&(iii)&(iv)\\
\scalebox{0.5}{
\begin{tikzpicture}
\fill[pattern={Hatch}, pattern color=gray!40] (0.5,0) rectangle (1,3);
\draw[dashed] (0.5,0)--(1,0);
\draw[dashed] (0.5,3)--(1,3);
\draw (3,3)--(6,3)--(6,0)--(1,0);
\fill[gray] (1,1) rectangle (3,3);
\fill[black] (3,2) rectangle (4,3);
\fill[pattern={Hatch}, pattern color=blue] (1,0) rectangle (2,1);
\fill[blue] (2,0) rectangle (3,1);
\end{tikzpicture}}&\scalebox{0.5}{
\begin{tikzpicture}
\fill[pattern={Hatch}, pattern color=gray!40] (0.5,0) rectangle (1,3);
\draw[dashed] (0.5,0)--(1,0);
\draw[dashed] (0.5,3)--(1,3);
\draw (3,3)--(6,3)--(6,0)--(1,0);
\fill[gray] (1,1) rectangle (3,3);
\fill[pattern={Hatch}, pattern color=black] (3,2) rectangle (4,3);
\fill[black] (3,1) rectangle (4,2);
\fill[blue] (1,0) rectangle (2,1);
\end{tikzpicture}}&\scalebox{0.5}{
\begin{tikzpicture}
\fill[pattern={Hatch}, pattern color=gray!40] (0.5,0) rectangle (1,3);
\draw[dashed] (0.5,0)--(1,0);
\draw[dashed] (0.5,3)--(1,3);
\draw (3,3)--(6,3)--(6,0)--(1,0);
\fill[gray] (1,1) rectangle (3,3);
\fill[black] (1,0) rectangle (2,1);
\fill[pattern={Hatch}, pattern color=blue] (3,2) rectangle (4,3);
\fill[blue] (3,1) rectangle (4,2);
\end{tikzpicture}}&\scalebox{0.5}{
\begin{tikzpicture}
\fill[pattern={Hatch}, pattern color=gray!40] (0.5,0) rectangle (1,3);
\draw[dashed] (0.5,0)--(1,0);
\draw[dashed] (0.5,3)--(1,3);
\draw (3,3)--(6,3)--(6,0)--(1,0);
\fill[gray] (1,1) rectangle (3,3);
\fill[pattern={Hatch}, pattern color=black] (1,0) rectangle (2,1);
\fill[black] (2,0) rectangle (3,1);
\fill[blue] (3,2) rectangle (4,3);
\end{tikzpicture}}
\end{tabular}
\end{center}

In (ii) and (iii), player 1's response results in a board with $k+1$ eliminated cells in the first two rows and $k-1$ eliminated cells in the third row, which matches case 1.  In (i) and (iv), player 1's response results in a board with $k+1$ eliminated cells in the first row and $k$ eliminated cells in rows 2 and 3, which matches case 2.

Similarly, in the second case, player 2 has three available moves: (i) $(k+1,1)$, (ii) $(k,2)$, or (iii) $(k,3)$, shown below in black.

\begin{center}
\begin{tabular}{ccc}
(i)&(ii)&(iii)\\
\scalebox{0.5}{\begin{tikzpicture}
\fill[pattern={Hatch}, pattern color=gray!40] (0.5,0) rectangle (1,3);
\draw[dashed] (0.5,0)--(1,0);
\draw[dashed] (0.5,3)--(1,3);
\draw (2,3)--(6,3)--(6,0)--(1,0);
\fill[gray] (1,2) rectangle (2,3);
\fill[black] (2,2) rectangle (3,3);
\end{tikzpicture}}&\scalebox{0.5}{\begin{tikzpicture}
\fill[pattern={Hatch}, pattern color=gray!40] (0.5,0) rectangle (1,3);
\draw[dashed] (0.5,0)--(1,0);
\draw[dashed] (0.5,3)--(1,3);
\draw (2,3)--(6,3)--(6,0)--(1,0);
\fill[gray] (1,2) rectangle (2,3);
\fill[black] (1,1) rectangle (2,2);
\end{tikzpicture}}&\scalebox{0.5}{\begin{tikzpicture}
\fill[pattern={Hatch}, pattern color=gray!40] (0.5,0) rectangle (1,3);
\draw[dashed] (0.5,0)--(1,0);
\draw[dashed] (0.5,3)--(1,3);
\draw (2,3)--(6,3)--(6,0)--(1,0);
\fill[gray] (1,2) rectangle (2,3);
\fill[pattern={Hatch}, pattern color=black] (1,1) rectangle (2,2);
\fill[black] (1,0) rectangle (2,1);
\end{tikzpicture}}
\end{tabular}
\end{center}

In each of these cases, player 1 again has a natural response, shown below in blue.

\begin{center}
\begin{tabular}{ccc}
(i)&(ii)&(iii)\\
\scalebox{0.5}{\begin{tikzpicture}
\fill[pattern={Hatch}, pattern color=gray!40] (0.5,0) rectangle (1,3);
\draw[dashed] (0.5,0)--(1,0);
\draw[dashed] (0.5,3)--(1,3);
\draw (2,3)--(6,3)--(6,0)--(1,0);
\fill[gray] (1,2) rectangle (2,3);
\fill[black] (2,2) rectangle (3,3);
\fill[pattern={Hatch}, pattern color=blue] (1,1) rectangle (2,2);
\fill[blue] (1,0) rectangle (2,1);
\end{tikzpicture}}&\scalebox{0.5}{\begin{tikzpicture}
\fill[pattern={Hatch}, pattern color=gray!40] (0.5,0) rectangle (1,3);
\draw[dashed] (0.5,0)--(1,0);
\draw[dashed] (0.5,3)--(1,3);
\draw (2,3)--(6,3)--(6,0)--(1,0);
\fill[gray] (1,2) rectangle (2,3);
\fill[black] (1,1) rectangle (2,2);
\fill[pattern={Hatch}, pattern color=blue] (2,2) rectangle (3,3);
\fill[blue] (2,1) rectangle (3,2);
\end{tikzpicture}}&\scalebox{0.5}{\begin{tikzpicture}
\fill[pattern={Hatch}, pattern color=gray!40] (0.5,0) rectangle (1,3);
\draw[dashed] (0.5,0)--(1,0);
\draw[dashed] (0.5,3)--(1,3);
\draw (2,3)--(6,3)--(6,0)--(1,0);
\fill[gray] (1,2) rectangle (2,3);
\fill[pattern={Hatch}, pattern color=black] (1,1) rectangle (2,2);
\fill[black] (1,0) rectangle (2,1);
\fill[blue] (2,2) rectangle (3,3);
\end{tikzpicture}}
\end{tabular}
\end{center}

In (ii), player 1's response results in a board with $k+1$ eliminated cells in the first two rows and $k-1$ eliminated cells in the third row, which matches case 1.  In (i) and (iii), player 1's response results in a board with $k+1$ eliminated cells in the first row and $k$ eliminated cells in rows 2 and 3, which matches case 2.

The endgame begins at the start of player 2's turn when row 1 has $a-2$ eliminated cells, shown in two cases below.

\begin{center}
\begin{tabular}{cc}
(i)&(ii)\\
\scalebox{0.5}{
\begin{tikzpicture}
\draw (0,0) -- (0,3)--(5,3)--(5,0)--(0,0);
\draw (-1,0)--(0,0);
\draw (-1,3)--(0,3);
\draw[step=1.0,black,thin] (0,0) grid (5,3);
\fill[gray] (-1,1) rectangle (4,3);
\fill[gray] (-1,0) rectangle (2,1);
\end{tikzpicture}}&\scalebox{0.5}{\begin{tikzpicture}
\draw (0,0) -- (0,3)--(5,3)--(5,0)--(0,0);
\draw (-1,0)--(0,0);
\draw (-1,3)--(0,3);
\draw[step=1.0,black,thin] (0,0) grid (5,3);
\fill[gray] (-1,2) rectangle (4,3);
\fill[gray] (-1,0) rectangle (3,2);
\end{tikzpicture}}\\
\end{tabular}
\end{center}

In case (i), player 2's options are $(a-1,1)$, $(a-1,2)$, $(a-3,3)$, or $(a-2,3)$.  If player 2 plays $(a-1,2)$ or $(a-2,3)$, player 1 can play $(a-1,3)$ and win the game.  So player 2 will play $(a-1,1)$ or  $(a-3,3)$.  Whichever of these two options player 2 takes, player 1 takes the other move.  This forces player 2 to play $(a-1,2)$ or $(a-2,3)$ on their next move, and player 1 takes the corner cell of $(a-1,3)$, forcing a player 2 loss.

In case (ii) player 2's options are $(a-1,1)$, $(a-2,2)$, or $(a-2,3)$.  If player 2 plays $(a-2,3)$, player 1 can play $(a-1,3)$ and win the game, or player 2 will play $(a-1,1)$ or $(a-2,2)$.  In either case, player 1 plays the other move.  This forces player 2 to play $(a-1,2)$ or $(a-2,3)$ on their next move, and player 1 takes the corner cell of $(a-1,3)$, forcing a player 2 loss.
\end{proof}

\section{Strategy when b=5}\label{S:5}

Finally, we consider the $(a,5)$-game for $a \geq 5$.  Although we give a winning strategy for this situation, there are more options for how player 1 selects a move leading to the final endgame.  The strategy given in this section was determined in an experimental manner.  Since this is a two-player combinatorial game with perfect information, when $a$ is known, every possible game state can be labeled as a next-player win (in class $\mathcal{N}$) or a next-player loss (in class $\mathcal{P}$) via computer search.  Although we are limited to small values of $a$ for a complete computer analysis, once the computer makes this labeling of all states for several specific values of $a$, an interested human can use the computer data to conjecture a subset of positions that are in class $\mathcal{P}$ regardless of $a$ and form a strategy that guarantees that regardless of player 2's move, player 1 can respond in such a way to achieve a position in the subset.  As in the previous section, there are three clear phases of the game: opening moves, midgame, and endgame.

\begin{theorem}\label{T:fivewin}
Player 1 has a winning strategy in the $(a,5)$-game where $a \geq 5$.
\end{theorem}

Before we formally prove Theorem \ref{T:fivewin}, we outline player 1's strategy.

For the midgame, player 1 acts to leave the board in one of the following seven states after their turn:
\begin{enumerate}
\item Row 1 has an odd number of open cells and $k$ eliminated cells, while rows 2, 3, and 4 have $k-1$ eliminated cells where $k \geq 2$.

\begin{center}
\scalebox{0.4}{\begin{tikzpicture}
\draw (0,0)--(1,0)--(1,3)--(2,3)--(2,4)--(0,4)--(0,0);
\draw (0,4)--(5,4)--(5,0)--(0,0);
\draw[step=1.0,black,thin] (0,0) grid (2,4);
\fill[gray] (0,0) rectangle (1,4);
\fill[gray] (1,3) rectangle (2,4);
\node at (3,3.5) {\Huge odd};
\end{tikzpicture}}
\end{center}

\item Rows 1, 2, and 3 have an even number of open cells and $k$ eliminated cells, while row 4 has $k-1$ eliminated cells where $k \geq 1$.

\begin{center}
\scalebox{0.4}{\begin{tikzpicture}
\draw (0,0)--(1,0)--(1,1)--(2,1)--(2,4)--(0,4)--(0,0);
\draw (0,4)--(5,4)--(5,0)--(0,0);
\draw[step=1.0,black,thin] (0,0) grid (2,4);
\fill[gray] (0,0) rectangle (1,4);
\fill[gray] (1,1) rectangle (2,4);
\node at (3,3.5) {\Huge even};
\end{tikzpicture}}
\end{center}

\item Rows 1 and 2 have an even number of open cells and $k$ eliminated cells, while row 3 has $k-1$ eliminated cells, and row 4 has $k-3$ eliminated cells where $k \geq 3$.

\begin{center}
\scalebox{0.4}{\begin{tikzpicture}
\draw (0,0)--(1,0)--(1,1)--(3,1)--(3,2)--(4,2)--(4,4)--(0,4)--(0,0);
\draw (0,4)--(7,4)--(7,0)--(0,0);
\draw[step=1.0,black,thin] (0,0) grid (4,4);
\fill[gray] (0,0) rectangle (1,4);
\fill[gray] (1,1) rectangle (3,4);
\fill[gray] (3,2) rectangle (4,4);
\node at (5,3.5) {\Huge even};
\end{tikzpicture}}
\end{center}

\item Row 1 has an even number of open cells and $k$ eliminated cells, rows 2 and 3 have $k-1$ eliminated cells, and row 4 has $k-2$ eliminated cells where $k \geq 2$.

\begin{center}
\scalebox{0.4}{\begin{tikzpicture}
\draw (0,0)--(1,0)--(1,1)--(2,1)--(2,3)--(3,3)--(3,4)--(0,4)--(0,0);
\draw (0,4)--(6,4)--(6,0)--(0,0);
\draw[step=1.0,black,thin] (0,0) grid (3,4);
\fill[gray] (0,0) rectangle (1,4);
\fill[gray] (1,1) rectangle (2,4);
\fill[gray] (2,3) rectangle (3,4);
\node at (4,3.5) {\Huge even};
\end{tikzpicture}} 
\end{center}

\item Rows 1 and 2 have $k$ eliminated cells and non-zero open cells, while rows 3 and 4 have $k-2$ eliminated cells where $k \geq 3$.

\begin{center}
\scalebox{0.4}{\begin{tikzpicture}
\draw (0,0)--(1,0)--(1,2)--(3,2)--(3,4)--(0,4)--(0,0);
\draw (0,4)--(5,4)--(5,0)--(0,0);
\draw[step=1.0,black,thin] (0,0) grid (3,4);
\fill[gray] (0,0) rectangle (1,4);
\fill[gray] (1,2) rectangle (2,4);
\fill[gray] (2,2) rectangle (3,4);
\end{tikzpicture}}
\end{center}

\item Row 1 has $k$ eliminated cells and non-zero open cells, row 2 has $k-1$ such cells, row 3 has $k-2$ such cells, and row 4 has $k-3$ such cells where $k \geq 3$.

\begin{center}\scalebox{0.4}{\begin{tikzpicture}
\draw (0,0)--(1,0)--(1,1)--(2,1)--(2,2)--(3,2)--(3,3)--(4,3)--(4,4)--(0,4)--(0,0);
\draw (0,4)--(6,4)--(6,0)--(0,0);
\draw[step=1.0,black,thin] (0,0) grid (4,4);
\fill[gray] (0,0) rectangle (1,4);
\fill[gray] (1,1) rectangle (2,4);
\fill[gray] (2,2) rectangle (3,4);
\fill[gray] (3,3) rectangle (4,4);
\end{tikzpicture}}
\end{center}

\item Row 1 has one more eliminated cell than row 2 and non-zero open cells.  Row 2 has at least two more eliminated cells than row 3, and row 3 has one more eliminated cell than row 4.

\begin{center}\scalebox{0.4}{\begin{tikzpicture}
\draw (0,0)--(1,0)--(1,1)--(2,1)--(2,2)--(5,2)--(5,3)--(6,3)--(6,4)--(0,4)--(0,0);
\draw (0,4)--(8,4)--(8,0)--(0,0);
\draw[step=1.0,black,thin] (0,0) grid (6,4);
\fill[white] (3.1,0.1) rectangle (3.9,3.9);
\fill[gray] (4,3) rectangle (6,4);
\fill[gray] (4,2) rectangle (5,3);
\fill[gray] (0,3) rectangle (3,4);
\fill[gray] (0,2) rectangle (3,3);
\fill[gray] (0,1) rectangle (2,2);
\fill[gray] (0,0) rectangle (1,1);
\node at (3.5,3) {$\cdots$};
\end{tikzpicture}}
\end{center}
\end{enumerate}

We refer to these seven states as the set $\widehat{S}$.  A tedious computer-assisted analysis shows that if the current board is in one of the states from $\widehat{S}$ at the start of player 2's turn, then no matter where player 2 moves, player 1 has a response that returns to a state in $\widehat{S}$.

Further, it is possible to navigate from any of these states to having one of the following three states to prepare for an endgame:

\begin{enumerate}
\item Both columns $a-1$ and $a-2$ have open cells.  Column $a-1$ has one more open cell than column $a-2$.  All other columns are completely eliminated.

\begin{center}
\scalebox{0.4}{\begin{tikzpicture}
\draw (0,0) -- (0,4)--(5,4)--(5,0)--(0,0);
\draw (-1,0)--(0,0);
\draw (-1,4)--(0,4);
\draw[step=1.0,black,thin] (0,0) grid (5,4);
\fill[gray] (-1,3) rectangle (4,4);
\fill[gray] (-1,0) rectangle (3,3);
\end{tikzpicture}}
\end{center}

\item Column $a-1$ has at least two open cells; row 4 has the same number of open cells.  All other cells are are eliminated.

\begin{center}
\scalebox{0.4}{\begin{tikzpicture}
\draw (0,0) -- (0,4)--(5,4)--(5,0)--(0,0);
\draw (-1,0)--(0,0);
\draw (-1,4)--(0,4);
\draw[step=1.0,black,thin] (0,0) grid (5,4);
\fill[gray] (-1,1) rectangle (4,4);
\fill[gray] (-1,0) rectangle (1,1);
\end{tikzpicture}} 
\end{center}

\item Rows 1 and 2 have no open cells.   Both rows 3 and 4 have open cells.  Row 4 has one more open cell than row 3.

\begin{center}
\scalebox{0.4}{\begin{tikzpicture}
\draw (0,0) -- (0,4)--(5,4)--(5,0)--(0,0);
\draw (-1,0)--(0,0);
\draw (-1,4)--(0,4);
\draw[step=1.0,black,thin] (0,0) grid (5,4);
\fill[gray] (-1,2) rectangle (5,4);
\fill[gray] (-1,1) rectangle (2,2);
\fill[gray] (-1,0) rectangle (1,1);
\end{tikzpicture}} 
\end{center}
\end{enumerate}

We refer to this set of three states as set $\widehat{E}$.

We will ultimately prove Theorem \ref{T:fivewin} by a sequence of lemmas.

\begin{lemma}\label{L:start}
By the end of their fourth move in an $(a,5)$-game with $a \geq 5$, player 1 can leave the board in a state from $\widehat{S}$.
\end{lemma}

\begin{proof}
Player 1's first move is $(1,1)$.  Regardless of whether player 2 chooses $(1,2)$ or $(2,1)$, player 1's second move is $(2,2)$ which results in a $2 \times 2$ eliminated region.

If player 2 plays cell $(3,1)$ or $(1,3)$, then player 1 responds by playing the other option which leaves the board in State 6 from $\widehat{S}$ after player 1's third move.

On the other hand, suppose player 2 plays $(3,2)$ or $(2,3)$ as their second move. We proceed in cases.

If $a-1$ is odd, then player 1 plays whichever of $(3,2)$ and $(2,3)$ was not chosen by player 2.  This results in State 3 from $\widehat{S}$.

If $a-1$ is even, then player 1 plays $(3,3)$ which results in a $3\times 3$ eliminated region.  Player 2 has six choices for what they can play in response: $(4,1)$, $(4,2)$, $(4,3)$, $(1,4)$, $(2,4)$, or $(3,4)$.  If $a=5$, the board is already in a state from $\widehat{E}$ and player 1 responds in a tit-for-tat way described in Lemma \ref{L:end}.  Otherwise, player 1 responds as described below.

If player 2 plays $(4,1)$ or $(2,4)$, player 1 plays the other of these two cells, which results in State 4 from $\widehat{S}$ on player 1's fourth move.

If player 2 plays $(4,2)$ or $(1,4)$, player 1 plays the other of these two cells, which results in State 3 from $\widehat{S}$ on player 1's fourth move.

If player 2 plays $(4,3)$ or $(3,4)$, player 1 plays the other of these two cells, which results in State 2 from $\widehat{S}$ on player 1's fourth move.
\end{proof}

Next, we proceed to the midgame, by showing player 1 has a strategy to return the game to a state from $\widehat{S}$ no matter how player 2 acts.  The proof we give documents the result of computer search in a manner that is admittedly tedious, but able to be verified by an interested reader.

\begin{lemma}\label{L:midgame}
If the board is in a state from $\widehat{S}$ at the end of player 1's turn, regardless of player 2's next move, player 1 can return the game to a state from $\widehat{S}$.
\end{lemma}

\begin{proof}
We proceed in cases by considering each of the seven states in $\widehat{S}$, each of player 2's options, and an appropriate response from player 1.

In State 1, player 2 has four possible moves, shown below, followed by player 1's responses.

\begin{center}
\begin{tabular}{cccc}
\scalebox{0.4}{\begin{tikzpicture}
\draw (0,0)--(1,0)--(1,3)--(2,3)--(2,4)--(0,4)--(0,0);
\draw (0,4)--(5,4)--(5,0)--(0,0);
\draw[step=1.0,black,thin] (0,0) grid (2,4);
\fill[gray] (0,0) rectangle (1,4);
\fill[gray] (1,3) rectangle (2,4);
\fill[black] (2,3) rectangle (3,4);
\node at (4,3.5) {\Huge even};
\end{tikzpicture}}
&
\scalebox{0.4}{\begin{tikzpicture}
\draw (0,0)--(1,0)--(1,3)--(2,3)--(2,4)--(0,4)--(0,0);
\draw (0,4)--(5,4)--(5,0)--(0,0);
\draw[step=1.0,black,thin] (0,0) grid (2,4);
\fill[gray] (0,0) rectangle (1,4);
\fill[gray] (1,3) rectangle (2,4);
\fill[black] (1,2) rectangle (2,3);
\node at (3,3.5) {\Huge odd};
\end{tikzpicture}}
&
\scalebox{0.4}{\begin{tikzpicture}
\draw (0,0)--(1,0)--(1,3)--(2,3)--(2,4)--(0,4)--(0,0);
\draw (0,4)--(5,4)--(5,0)--(0,0);
\draw[step=1.0,black,thin] (0,0) grid (2,4);
\fill[gray] (0,0) rectangle (1,4);
\fill[gray] (1,3) rectangle (2,4);
\fill[pattern={Hatch}, pattern color=black] (1,2) rectangle (2,3);
\fill[black] (1,1) rectangle (2,2);
\node at (3,3.5) {\Huge odd};
\end{tikzpicture}}
&
\scalebox{0.4}{\begin{tikzpicture}
\draw (0,0)--(1,0)--(1,3)--(2,3)--(2,4)--(0,4)--(0,0);
\draw (0,4)--(5,4)--(5,0)--(0,0);
\draw[step=1.0,black,thin] (0,0) grid (2,4);
\fill[gray] (0,0) rectangle (1,4);
\fill[gray] (1,3) rectangle (2,4);
\fill[pattern={Hatch}, pattern color=black] (1,1) rectangle (2,3);
\fill[black] (1,0) rectangle (2,1);
\node at (3,3.5) {\Huge odd};
\end{tikzpicture}}
\\
\scalebox{0.4}{\begin{tikzpicture}
\draw (0,0)--(1,0)--(1,3)--(2,3)--(2,4)--(0,4)--(0,0);
\draw (0,4)--(5,4)--(5,0)--(0,0);
\draw[step=1.0,black,thin] (0,0) grid (2,4);
\fill[gray] (0,0) rectangle (1,4);
\fill[gray] (1,3) rectangle (2,4);
\fill[black] (2,3) rectangle (3,4);
\fill[pattern={Hatch}, pattern color=blue] (1,2) rectangle (2,3);
\fill[blue] (2,2) rectangle (3,3);
\end{tikzpicture}}
&
\scalebox{0.4}{\begin{tikzpicture}
\draw (0,0)--(1,0)--(1,3)--(2,3)--(2,4)--(0,4)--(0,0);
\draw (0,4)--(5,4)--(5,0)--(0,0);
\draw[step=1.0,black,thin] (0,0) grid (2,4);
\fill[gray] (0,0) rectangle (1,4);
\fill[gray] (1,3) rectangle (2,4);
\fill[black] (1,2) rectangle (2,3);
\fill[pattern={Hatch}, pattern color=blue] (2,3) rectangle (3,4);
\fill[blue] (2,2) rectangle (3,3);
\end{tikzpicture}}
&
\scalebox{0.4}{\begin{tikzpicture}
\draw (0,0)--(1,0)--(1,3)--(2,3)--(2,4)--(0,4)--(0,0);
\draw (0,4)--(5,4)--(5,0)--(0,0);
\draw[step=1.0,black,thin] (0,0) grid (2,4);
\fill[gray] (0,0) rectangle (1,4);
\fill[gray] (1,3) rectangle (2,4);
\fill[pattern={Hatch}, pattern color=black] (1,2) rectangle (2,3);
\fill[black] (1,1) rectangle (2,2);
\fill[blue] (2,3) rectangle (3,4);
\node at (4,3.5) {\Huge even};
\end{tikzpicture}}
&
\scalebox{0.4}{\begin{tikzpicture}
\draw (0,0)--(1,0)--(1,3)--(2,3)--(2,4)--(0,4)--(0,0);
\draw (0,4)--(5,4)--(5,0)--(0,0);
\draw[step=1.0,black,thin] (0,0) grid (2,4);
\fill[gray] (0,0) rectangle (1,4);
\fill[gray] (1,3) rectangle (2,4);
\fill[pattern={Hatch}, pattern color=black] (1,1) rectangle (2,3);
\fill[black] (1,0) rectangle (2,1);
\fill[pattern={Hatch}, pattern color=blue] (2,2) rectangle (3,4);
\fill[blue] (2,1) rectangle (3,2);
\node at (4,3.5) {\Huge even};
\end{tikzpicture}}\\
State 5&State 5&State 4&State 2\\
\end{tabular}
\end{center}

In State 2, player 2 has four possible moves, shown below, followed by player 1's responses.
\begin{center}
\begin{tabular}{cccc}
\scalebox{0.4}{\begin{tikzpicture}
\draw (0,0)--(1,0)--(1,1)--(2,1)--(2,4)--(0,4)--(0,0);
\draw (0,4)--(5,4)--(5,0)--(0,0);
\draw[step=1.0,black,thin] (0,0) grid (2,4);
\fill[gray] (0,0) rectangle (1,4);
\fill[gray] (1,1) rectangle (2,4);
\fill[black] (2,3) rectangle (3,4);
\node at (4,3.5) {\Huge odd};
\end{tikzpicture}}&\scalebox{0.4}{\begin{tikzpicture}
\draw (0,0)--(1,0)--(1,1)--(2,1)--(2,4)--(0,4)--(0,0);
\draw (0,4)--(5,4)--(5,0)--(0,0);
\draw[step=1.0,black,thin] (0,0) grid (2,4);
\fill[gray] (0,0) rectangle (1,4);
\fill[gray] (1,1) rectangle (2,4);
\fill[pattern={Hatch}, pattern color=black] (2,3) rectangle (3,4);
\fill[black] (2,2) rectangle (3,3);
\node at (4,3.5) {\Huge odd};
\end{tikzpicture}}&\scalebox{0.4}{\begin{tikzpicture}
\draw (0,0)--(1,0)--(1,1)--(2,1)--(2,4)--(0,4)--(0,0);
\draw (0,4)--(5,4)--(5,0)--(0,0);
\draw[step=1.0,black,thin] (0,0) grid (2,4);
\fill[gray] (0,0) rectangle (1,4);
\fill[gray] (1,1) rectangle (2,4);
\fill[pattern={Hatch}, pattern color=black] (2,2) rectangle (3,4);
\fill[black] (2,1) rectangle (3,2);
\node at (4,3.5) {\Huge odd};
\end{tikzpicture}}&\scalebox{0.4}{\begin{tikzpicture}
\draw (0,0)--(1,0)--(1,1)--(2,1)--(2,4)--(0,4)--(0,0);
\draw (0,4)--(5,4)--(5,0)--(0,0);
\draw[step=1.0,black,thin] (0,0) grid (2,4);
\fill[gray] (0,0) rectangle (1,4);
\fill[gray] (1,1) rectangle (2,4);
\fill[black] (1,0) rectangle (2,1);
\node at (3,3.5) {\Huge even};
\end{tikzpicture}}\\
\scalebox{0.4}{\begin{tikzpicture}
\draw (0,0)--(1,0)--(1,1)--(2,1)--(2,4)--(0,4)--(0,0);
\draw (0,4)--(5,4)--(5,0)--(0,0);
\draw[step=1.0,black,thin] (0,0) grid (2,4);
\fill[gray] (0,0) rectangle (1,4);
\fill[gray] (1,1) rectangle (2,4);
\fill[black] (2,3) rectangle (3,4);
\fill[blue] (1,0) rectangle (2,1);
\node at (4,3.5) {\Huge odd};
\end{tikzpicture}}&\scalebox{0.4}{\begin{tikzpicture}
\draw (0,0)--(1,0)--(1,1)--(2,1)--(2,4)--(0,4)--(0,0);
\draw (0,4)--(5,4)--(5,0)--(0,0);
\draw[step=1.0,black,thin] (0,0) grid (2,4);
\fill[gray] (0,0) rectangle (1,4);
\fill[gray] (1,1) rectangle (2,4);
\fill[pattern={Hatch}, pattern color=black] (2,3) rectangle (3,4);
\fill[black] (2,2) rectangle (3,3);
\fill[blue] (3,3) rectangle (4,4);
\end{tikzpicture}}&\scalebox{0.4}{\begin{tikzpicture}
\draw (0,0)--(1,0)--(1,1)--(2,1)--(2,4)--(0,4)--(0,0);
\draw (0,4)--(6,4)--(6,0)--(0,0);
\draw[step=1.0,black,thin] (0,0) grid (2,4);
\fill[gray] (0,0) rectangle (1,4);
\fill[gray] (1,1) rectangle (2,4);
\fill[pattern={Hatch}, pattern color=black] (2,2) rectangle (3,4);
\fill[black] (2,1) rectangle (3,2);
\fill[pattern={Hatch}, pattern color=blue] (3,3) rectangle (4,4);
\fill[blue] (3,2) rectangle (4,3);
\node at (5,3.5) {\Huge even};
\end{tikzpicture}}&\scalebox{0.4}{\begin{tikzpicture}
\draw (0,0)--(1,0)--(1,1)--(2,1)--(2,4)--(0,4)--(0,0);
\draw (0,4)--(5,4)--(5,0)--(0,0);
\draw[step=1.0,black,thin] (0,0) grid (2,4);
\fill[gray] (0,0) rectangle (1,4);
\fill[gray] (1,1) rectangle (2,4);
\fill[black] (1,0) rectangle (2,1);
\fill[blue] (2,3) rectangle (3,4);
\node at (4,3.5) {\Huge odd};
\end{tikzpicture}}\\
State 1& State 6 & State 3&State 1
\end{tabular}
\end{center}

In State 3, player 2 has five possible moves, shown below, followed by player 1's responses.

\begin{center}
\begin{tabular}{ccccc}
\scalebox{0.28}{\begin{tikzpicture}
\draw (0,0)--(1,0)--(1,1)--(3,1)--(3,2)--(4,2)--(4,4)--(0,4)--(0,0);
\draw (0,4)--(7,4)--(7,0)--(0,0);
\draw[step=1.0,black,thin] (0,0) grid (4,4);
\fill[gray] (0,0) rectangle (1,4);
\fill[gray] (1,1) rectangle (3,4);
\fill[gray] (3,2) rectangle (4,4);
\fill[black] (4,3) rectangle (5,4);
\node at (6,3.5) {\Huge odd};
\end{tikzpicture}}&\scalebox{0.28}{\begin{tikzpicture}
\draw (0,0)--(1,0)--(1,1)--(3,1)--(3,2)--(4,2)--(4,4)--(0,4)--(0,0);
\draw (0,4)--(7,4)--(7,0)--(0,0);
\draw[step=1.0,black,thin] (0,0) grid (4,4);
\fill[gray] (0,0) rectangle (1,4);
\fill[gray] (1,1) rectangle (3,4);
\fill[gray] (3,2) rectangle (4,4);
\fill[pattern={Hatch}, pattern color=black] (4,3) rectangle (5,4);
\fill[black] (4,2) rectangle (5,3);
\node at (6,3.5) {\Huge odd};
\end{tikzpicture}}&\scalebox{0.28}{\begin{tikzpicture}
\draw (0,0)--(1,0)--(1,1)--(3,1)--(3,2)--(4,2)--(4,4)--(0,4)--(0,0);
\draw (0,4)--(7,4)--(7,0)--(0,0);
\draw[step=1.0,black,thin] (0,0) grid (4,4);
\fill[gray] (0,0) rectangle (1,4);
\fill[gray] (1,1) rectangle (3,4);
\fill[gray] (3,2) rectangle (4,4);
\fill[black] (3,1) rectangle (4,2);
\node at (5,3.5) {\Huge even};
\end{tikzpicture}}&\scalebox{0.28}{\begin{tikzpicture}
\draw (0,0)--(1,0)--(1,1)--(3,1)--(3,2)--(4,2)--(4,4)--(0,4)--(0,0);
\draw (0,4)--(7,4)--(7,0)--(0,0);
\draw[step=1.0,black,thin] (0,0) grid (4,4);
\fill[gray] (0,0) rectangle (1,4);
\fill[gray] (1,1) rectangle (3,4);
\fill[gray] (3,2) rectangle (4,4);
\fill[pattern={Hatch}, pattern color=black] (1,0) rectangle (2,1);
\fill[black] (2,0) rectangle (3,1);
\node at (5,3.5) {\Huge even};
\end{tikzpicture}}&\scalebox{0.28}{\begin{tikzpicture}
\draw (0,0)--(1,0)--(1,1)--(3,1)--(3,2)--(4,2)--(4,4)--(0,4)--(0,0);
\draw (0,4)--(7,4)--(7,0)--(0,0);
\draw[step=1.0,black,thin] (0,0) grid (4,4);
\fill[gray] (0,0) rectangle (1,4);
\fill[gray] (1,1) rectangle (3,4);
\fill[gray] (3,2) rectangle (4,4);
\fill[black] (1,0) rectangle (2,1);
\node at (5,3.5) {\Huge even};
\end{tikzpicture}}\\
\scalebox{0.28}{\begin{tikzpicture}
\draw (0,0)--(1,0)--(1,1)--(3,1)--(3,2)--(4,2)--(4,4)--(0,4)--(0,0);
\draw (0,4)--(7,4)--(7,0)--(0,0);
\draw[step=1.0,black,thin] (0,0) grid (4,4);
\fill[gray] (0,0) rectangle (1,4);
\fill[gray] (1,1) rectangle (3,4);
\fill[gray] (3,2) rectangle (4,4);
\fill[black] (4,3) rectangle (5,4);
\fill[blue] (1,0) rectangle (2,1);
\end{tikzpicture}}&\scalebox{0.28}{\begin{tikzpicture}
\draw (0,0)--(1,0)--(1,1)--(3,1)--(3,2)--(4,2)--(4,4)--(0,4)--(0,0);
\draw (0,4)--(7,4)--(7,0)--(0,0);
\draw[step=1.0,black,thin] (0,0) grid (4,4);
\fill[gray] (0,0) rectangle (1,4);
\fill[gray] (1,1) rectangle (3,4);
\fill[gray] (3,2) rectangle (4,4);
\fill[pattern={Hatch}, pattern color=black] (4,3) rectangle (5,4);
\fill[black] (4,2) rectangle (5,3);
\fill[pattern={Hatch}, pattern color=blue] (1,0) rectangle (2,1);
\fill[blue] (2,0) rectangle (3,1);
\end{tikzpicture}}&\scalebox{0.28}{\begin{tikzpicture}
\draw (0,0)--(1,0)--(1,1)--(3,1)--(3,2)--(4,2)--(4,4)--(0,4)--(0,0);
\draw (0,4)--(7,4)--(7,0)--(0,0);
\draw[step=1.0,black,thin] (0,0) grid (4,4);
\fill[gray] (0,0) rectangle (1,4);
\fill[gray] (1,1) rectangle (3,4);
\fill[gray] (3,2) rectangle (4,4);
\fill[black] (3,1) rectangle (4,2);
\fill[pattern={Hatch}, pattern color=blue] (1,0) rectangle (2,1);
\fill[blue] (2,0) rectangle (3,1);
\node at (5,3.5) {\Huge even};
\end{tikzpicture}}&\scalebox{0.28}{\begin{tikzpicture}
\draw (0,0)--(1,0)--(1,1)--(3,1)--(3,2)--(4,2)--(4,4)--(0,4)--(0,0);
\draw (0,4)--(7,4)--(7,0)--(0,0);
\draw[step=1.0,black,thin] (0,0) grid (4,4);
\fill[gray] (0,0) rectangle (1,4);
\fill[gray] (1,1) rectangle (3,4);
\fill[gray] (3,2) rectangle (4,4);
\fill[pattern={Hatch}, pattern color=black] (1,0) rectangle (2,1);
\fill[black] (2,0) rectangle (3,1);
\fill[pattern={Hatch}, pattern color=blue] (4,3) rectangle (5,4);
\fill[blue] (4,2) rectangle (5,3);
\end{tikzpicture}}&\scalebox{0.28}{\begin{tikzpicture}
\draw (0,0)--(1,0)--(1,1)--(3,1)--(3,2)--(4,2)--(4,4)--(0,4)--(0,0);
\draw (0,4)--(7,4)--(7,0)--(0,0);
\draw[step=1.0,black,thin] (0,0) grid (4,4);
\fill[gray] (0,0) rectangle (1,4);
\fill[gray] (1,1) rectangle (3,4);
\fill[gray] (3,2) rectangle (4,4);
\fill[black] (1,0) rectangle (2,1);
\fill[blue] (4,3) rectangle (5,4);
\end{tikzpicture}}\\
State 6&State 5&State 2&State 5&State 6\\
\end{tabular}
\end{center}

In State 4, player 2 has four possible moves, shown below, followed by player 1's responses.

\begin{center}
\begin{tabular}{cccc}
\scalebox{0.4}{\begin{tikzpicture}
\draw (0,0)--(1,0)--(1,1)--(2,1)--(2,3)--(3,3)--(3,4)--(0,4)--(0,0);
\draw (0,4)--(6,4)--(6,0)--(0,0);
\draw[step=1.0,black,thin] (0,0) grid (3,4);
\fill[gray] (0,0) rectangle (1,4);
\fill[gray] (1,1) rectangle (2,4);
\fill[gray] (2,3) rectangle (3,4);
\fill[black] (3,3) rectangle (4,4);
\node at (5,3.5) {\Huge odd};
\end{tikzpicture}} &\scalebox{0.4}{\begin{tikzpicture}
\draw (0,0)--(1,0)--(1,1)--(2,1)--(2,3)--(3,3)--(3,4)--(0,4)--(0,0);
\draw (0,4)--(6,4)--(6,0)--(0,0);
\draw[step=1.0,black,thin] (0,0) grid (3,4);
\fill[gray] (0,0) rectangle (1,4);
\fill[gray] (1,1) rectangle (2,4);
\fill[gray] (2,3) rectangle (3,4);
\fill[black] (2,2) rectangle (3,3);
\node at (4,3.5) {\Huge even};
\end{tikzpicture}} &\scalebox{0.4}{\begin{tikzpicture}
\draw (0,0)--(1,0)--(1,1)--(2,1)--(2,3)--(3,3)--(3,4)--(0,4)--(0,0);
\draw (0,4)--(6,4)--(6,0)--(0,0);
\draw[step=1.0,black,thin] (0,0) grid (3,4);
\fill[gray] (0,0) rectangle (1,4);
\fill[gray] (1,1) rectangle (2,4);
\fill[gray] (2,3) rectangle (3,4);
\fill[pattern={Hatch}, pattern color=black] (2,2) rectangle (3,3);
\fill[black] (2,1) rectangle (3,2);
\node at (4,3.5) {\Huge even};
\end{tikzpicture}} &\scalebox{0.4}{\begin{tikzpicture}
\draw (0,0)--(1,0)--(1,1)--(2,1)--(2,3)--(3,3)--(3,4)--(0,4)--(0,0);
\draw (0,4)--(6,4)--(6,0)--(0,0);
\draw[step=1.0,black,thin] (0,0) grid (3,4);
\fill[gray] (0,0) rectangle (1,4);
\fill[gray] (1,1) rectangle (2,4);
\fill[gray] (2,3) rectangle (3,4);
\fill[black] (1,0) rectangle (2,1);
\node at (4,3.5) {\Huge even};
\end{tikzpicture}} \\
\scalebox{0.4}{\begin{tikzpicture}
\draw (0,0)--(1,0)--(1,1)--(2,1)--(2,3)--(3,3)--(3,4)--(0,4)--(0,0);
\draw (0,4)--(6,4)--(6,0)--(0,0);
\draw[step=1.0,black,thin] (0,0) grid (3,4);
\fill[gray] (0,0) rectangle (1,4);
\fill[gray] (1,1) rectangle (2,4);
\fill[gray] (2,3) rectangle (3,4);
\fill[black] (3,3) rectangle (4,4);
\fill[blue] (2,2) rectangle (3,3);
\end{tikzpicture}} &\scalebox{0.4}{\begin{tikzpicture}
\draw (0,0)--(1,0)--(1,1)--(2,1)--(2,3)--(3,3)--(3,4)--(0,4)--(0,0);
\draw (0,4)--(6,4)--(6,0)--(0,0);
\draw[step=1.0,black,thin] (0,0) grid (3,4);
\fill[gray] (0,0) rectangle (1,4);
\fill[gray] (1,1) rectangle (2,4);
\fill[gray] (2,3) rectangle (3,4);
\fill[black] (2,2) rectangle (3,3);
\fill[blue] (3,3) rectangle (4,4);
\end{tikzpicture}} &\scalebox{0.4}{\begin{tikzpicture}
\draw (0,0)--(1,0)--(1,1)--(2,1)--(2,3)--(3,3)--(3,4)--(0,4)--(0,0);
\draw (0,4)--(6,4)--(6,0)--(0,0);
\draw[step=1.0,black,thin] (0,0) grid (3,4);
\fill[gray] (0,0) rectangle (1,4);
\fill[gray] (1,1) rectangle (2,4);
\fill[gray] (2,3) rectangle (3,4);
\fill[pattern={Hatch}, pattern color=black] (2,2) rectangle (3,3);
\fill[black] (2,1) rectangle (3,2);
\fill[blue] (1,0) rectangle (2,1);
\node at (4,3.5) {\Huge even};
\end{tikzpicture}} &\scalebox{0.4}{\begin{tikzpicture}
\draw (0,0)--(1,0)--(1,1)--(2,1)--(2,3)--(3,3)--(3,4)--(0,4)--(0,0);
\draw (0,4)--(6,4)--(6,0)--(0,0);
\draw[step=1.0,black,thin] (0,0) grid (3,4);
\fill[gray] (0,0) rectangle (1,4);
\fill[gray] (1,1) rectangle (2,4);
\fill[gray] (2,3) rectangle (3,4);
\fill[black] (1,0) rectangle (2,1);
\fill[pattern={Hatch}, pattern color=blue] (2,2) rectangle (3,3);
\fill[blue] (2,1) rectangle (3,2);
\node at (4,3.5) {\Huge even};
\end{tikzpicture}} \\
State 6&State 6&State 2&State 2\\
\end{tabular}
\end{center}

In State 5, player 2 has five possible moves, shown below, followed by player 1's responses. In one of these cases, player 1's response depends on the parity of the number of open cells in row 1.

\begin{center}
\begin{tabular}{cccccc}
\scalebox{0.28}{\begin{tikzpicture}
\draw (0,0)--(1,0)--(1,2)--(3,2)--(3,4)--(0,4)--(0,0);
\draw (0,4)--(5,4)--(5,0)--(0,0);
\draw[step=1.0,black,thin] (0,0) grid (3,4);
\fill[gray] (0,0) rectangle (1,4);
\fill[gray] (1,2) rectangle (2,4);
\fill[gray] (2,2) rectangle (3,4);
\fill[black] (3,3) rectangle (4,4);
\end{tikzpicture}}&
\scalebox{0.28}{\begin{tikzpicture}
\draw (0,0)--(1,0)--(1,2)--(3,2)--(3,4)--(0,4)--(0,0);
\draw (0,4)--(5,4)--(5,0)--(0,0);
\draw[step=1.0,black,thin] (0,0) grid (3,4);
\fill[gray] (0,0) rectangle (1,4);
\fill[gray] (1,2) rectangle (2,4);
\fill[gray] (2,2) rectangle (3,4);
\fill[pattern={Hatch}, pattern color=black] (3,3) rectangle (4,4);
\fill[black] (3,2) rectangle (4,3);
\end{tikzpicture}}&\scalebox{0.28}{\begin{tikzpicture}
\draw (0,0)--(1,0)--(1,2)--(3,2)--(3,4)--(0,4)--(0,0);
\draw (0,4)--(5,4)--(5,0)--(0,0);
\draw[step=1.0,black,thin] (0,0) grid (3,4);
\fill[gray] (0,0) rectangle (1,4);
\fill[gray] (1,2) rectangle (2,4);
\fill[gray] (2,2) rectangle (3,4);
\fill[pattern={Hatch}, pattern color=black] (1,1) rectangle (2,2);
\fill[black] (2,1) rectangle (3,2);
\node at (4,3.5) {\Huge odd};
\end{tikzpicture}}&\scalebox{0.28}{\begin{tikzpicture}
\draw (0,0)--(1,0)--(1,2)--(3,2)--(3,4)--(0,4)--(0,0);
\draw (0,4)--(5,4)--(5,0)--(0,0);
\draw[step=1.0,black,thin] (0,0) grid (3,4);
\fill[gray] (0,0) rectangle (1,4);
\fill[gray] (1,2) rectangle (2,4);
\fill[gray] (2,2) rectangle (3,4);
\fill[pattern={Hatch}, pattern color=black] (1,1) rectangle (2,2);
\fill[black] (2,1) rectangle (3,2);
\node at (4,3.5) {\Huge even};
\end{tikzpicture}}&\scalebox{0.28}{\begin{tikzpicture}
\draw (0,0)--(1,0)--(1,2)--(3,2)--(3,4)--(0,4)--(0,0);
\draw (0,4)--(5,4)--(5,0)--(0,0);
\draw[step=1.0,black,thin] (0,0) grid (3,4);
\fill[gray] (0,0) rectangle (1,4);
\fill[gray] (1,2) rectangle (2,4);
\fill[gray] (2,2) rectangle (3,4);
\fill[black] (1,1) rectangle (2,2);
\end{tikzpicture}}&\scalebox{0.28}{\begin{tikzpicture}
\draw (0,0)--(1,0)--(1,2)--(3,2)--(3,4)--(0,4)--(0,0);
\draw (0,4)--(5,4)--(5,0)--(0,0);
\draw[step=1.0,black,thin] (0,0) grid (3,4);
\fill[gray] (0,0) rectangle (1,4);
\fill[gray] (1,2) rectangle (2,4);
\fill[gray] (2,2) rectangle (3,4);
\fill[pattern={Hatch}, pattern color=black] (1,1) rectangle (2,2);
\fill[black] (1,0) rectangle (2,1);
\end{tikzpicture}}\\
\scalebox{0.28}{\begin{tikzpicture}
\draw (0,0)--(1,0)--(1,2)--(3,2)--(3,4)--(0,4)--(0,0);
\draw (0,4)--(5,4)--(5,0)--(0,0);
\draw[step=1.0,black,thin] (0,0) grid (3,4);
\fill[gray] (0,0) rectangle (1,4);
\fill[gray] (1,2) rectangle (2,4);
\fill[gray] (2,2) rectangle (3,4);
\fill[black] (3,3) rectangle (4,4);
\fill[blue] (1,1) rectangle (2,2);
\end{tikzpicture}}&
\scalebox{0.28}{\begin{tikzpicture}
\draw (0,0)--(1,0)--(1,2)--(3,2)--(3,4)--(0,4)--(0,0);
\draw (0,4)--(5,4)--(5,0)--(0,0);
\draw[step=1.0,black,thin] (0,0) grid (3,4);
\fill[gray] (0,0) rectangle (1,4);
\fill[gray] (1,2) rectangle (2,4);
\fill[gray] (2,2) rectangle (3,4);
\fill[pattern={Hatch}, pattern color=black] (3,3) rectangle (4,4);
\fill[black] (3,2) rectangle (4,3);
\fill[pattern={Hatch}, pattern color=blue] (1,1) rectangle (2,2);
\fill[blue] (1,0) rectangle (2,1);
\end{tikzpicture}}&\scalebox{0.28}{\begin{tikzpicture}
\draw (0,0)--(1,0)--(1,2)--(3,2)--(3,4)--(0,4)--(0,0);
\draw (0,4)--(6,4)--(6,0)--(0,0);
\draw[step=1.0,black,thin] (0,0) grid (3,4);
\fill[gray] (0,0) rectangle (1,4);
\fill[gray] (1,2) rectangle (2,4);
\fill[gray] (2,2) rectangle (3,4);
\fill[pattern={Hatch}, pattern color=black] (1,1) rectangle (2,2);
\fill[black] (2,1) rectangle (3,2);
\fill[pattern={Hatch}, pattern color=blue] (3,3) rectangle (4,4);
\fill[blue] (3,2) rectangle (4,3);
\node at (5,3.5) {\Huge even};
\end{tikzpicture}}&\scalebox{0.28}{\begin{tikzpicture}
\draw (0,0)--(1,0)--(1,2)--(3,2)--(3,4)--(0,4)--(0,0);
\draw (0,4)--(5,4)--(5,0)--(0,0);
\draw[step=1.0,black,thin] (0,0) grid (3,4);
\fill[gray] (0,0) rectangle (1,4);
\fill[gray] (1,2) rectangle (2,4);
\fill[gray] (2,2) rectangle (3,4);
\fill[pattern={Hatch}, pattern color=black] (1,1) rectangle (2,2);
\fill[black] (2,1) rectangle (3,2);
\fill[blue] (1,0) rectangle (2,1);
\node at (4,3.5) {\Huge even};
\end{tikzpicture}}&\scalebox{0.28}{\begin{tikzpicture}
\draw (0,0)--(1,0)--(1,2)--(3,2)--(3,4)--(0,4)--(0,0);
\draw (0,4)--(5,4)--(5,0)--(0,0);
\draw[step=1.0,black,thin] (0,0) grid (3,4);
\fill[gray] (0,0) rectangle (1,4);
\fill[gray] (1,2) rectangle (2,4);
\fill[gray] (2,2) rectangle (3,4);
\fill[black] (1,1) rectangle (2,2);
\fill[blue] (3,3) rectangle (4,4);
\end{tikzpicture}}&\scalebox{0.28}{\begin{tikzpicture}
\draw (0,0)--(1,0)--(1,2)--(3,2)--(3,4)--(0,4)--(0,0);
\draw (0,4)--(5,4)--(5,0)--(0,0);
\draw[step=1.0,black,thin] (0,0) grid (3,4);
\fill[gray] (0,0) rectangle (1,4);
\fill[gray] (1,2) rectangle (2,4);
\fill[gray] (2,2) rectangle (3,4);
\fill[pattern={Hatch}, pattern color=black] (1,1) rectangle (2,2);
\fill[black] (1,0) rectangle (2,1);
\fill[pattern={Hatch}, pattern color=blue] (3,3) rectangle (4,4);
\fill[blue] (3,2) rectangle (4,3);
\end{tikzpicture}}\\
State 6&State 5&State 3&State 2&State 6&State 5\\
\end{tabular}
\end{center}

In State 6, player 2 has four possible moves, shown below, followed by player 1's responses. In one of these cases, player 1's response depends on the parity of the number of open cells in row 1.

\begin{center}
\begin{tabular}{ccccc}
\scalebox{0.28}{\begin{tikzpicture}
\draw (0,0)--(1,0)--(1,1)--(2,1)--(2,2)--(3,2)--(3,3)--(4,3)--(4,4)--(0,4)--(0,0);
\draw (0,4)--(6,4)--(6,0)--(0,0);
\draw[step=1.0,black,thin] (0,0) grid (4,4);
\fill[gray] (0,0) rectangle (1,4);
\fill[gray] (1,1) rectangle (2,4);
\fill[gray] (2,2) rectangle (3,4);
\fill[gray] (3,3) rectangle (4,4);
\fill[black] (4,3) rectangle (5,4);
\end{tikzpicture}}&\scalebox{0.28}{\begin{tikzpicture}
\draw (0,0)--(1,0)--(1,1)--(2,1)--(2,2)--(3,2)--(3,3)--(4,3)--(4,4)--(0,4)--(0,0);
\draw (0,4)--(6,4)--(6,0)--(0,0);
\draw[step=1.0,black,thin] (0,0) grid (4,4);
\fill[gray] (0,0) rectangle (1,4);
\fill[gray] (1,1) rectangle (2,4);
\fill[gray] (2,2) rectangle (3,4);
\fill[gray] (3,3) rectangle (4,4);
\fill[black] (3,2) rectangle (4,3);
\end{tikzpicture}}&\scalebox{0.28}{\begin{tikzpicture}
\draw (0,0)--(1,0)--(1,1)--(2,1)--(2,2)--(3,2)--(3,3)--(4,3)--(4,4)--(0,4)--(0,0);
\draw (0,4)--(6,4)--(6,0)--(0,0);
\draw[step=1.0,black,thin] (0,0) grid (4,4);
\fill[gray] (0,0) rectangle (1,4);
\fill[gray] (1,1) rectangle (2,4);
\fill[gray] (2,2) rectangle (3,4);
\fill[gray] (3,3) rectangle (4,4);
\fill[black] (2,1) rectangle (3,2);
\node at (5,3.5) {\Huge even};
\end{tikzpicture}}&\scalebox{0.28}{\begin{tikzpicture}
\draw (0,0)--(1,0)--(1,1)--(2,1)--(2,2)--(3,2)--(3,3)--(4,3)--(4,4)--(0,4)--(0,0);
\draw (0,4)--(6,4)--(6,0)--(0,0);
\draw[step=1.0,black,thin] (0,0) grid (4,4);
\fill[gray] (0,0) rectangle (1,4);
\fill[gray] (1,1) rectangle (2,4);
\fill[gray] (2,2) rectangle (3,4);
\fill[gray] (3,3) rectangle (4,4);
\fill[black] (2,1) rectangle (3,2);
\node at (5,3.5) {\Huge odd};
\end{tikzpicture}}&\scalebox{0.28}{\begin{tikzpicture}
\draw (0,0)--(1,0)--(1,1)--(2,1)--(2,2)--(3,2)--(3,3)--(4,3)--(4,4)--(0,4)--(0,0);
\draw (0,4)--(6,4)--(6,0)--(0,0);
\draw[step=1.0,black,thin] (0,0) grid (4,4);
\fill[gray] (0,0) rectangle (1,4);
\fill[gray] (1,1) rectangle (2,4);
\fill[gray] (2,2) rectangle (3,4);
\fill[gray] (3,3) rectangle (4,4);
\fill[black] (1,0) rectangle (2,1);
\end{tikzpicture}}\\
\scalebox{0.28}{\begin{tikzpicture}
\draw (0,0)--(1,0)--(1,1)--(2,1)--(2,2)--(3,2)--(3,3)--(4,3)--(4,4)--(0,4)--(0,0);
\draw (0,4)--(6,4)--(6,0)--(0,0);
\draw[step=1.0,black,thin] (0,0) grid (4,4);
\fill[gray] (0,0) rectangle (1,4);
\fill[gray] (1,1) rectangle (2,4);
\fill[gray] (2,2) rectangle (3,4);
\fill[gray] (3,3) rectangle (4,4);
\fill[black] (4,3) rectangle (5,4);
\fill[blue] (3,2) rectangle (4,3);
\end{tikzpicture}}&\scalebox{0.28}{\begin{tikzpicture}
\draw (0,0)--(1,0)--(1,1)--(2,1)--(2,2)--(3,2)--(3,3)--(4,3)--(4,4)--(0,4)--(0,0);
\draw (0,4)--(6,4)--(6,0)--(0,0);
\draw[step=1.0,black,thin] (0,0) grid (4,4);
\fill[gray] (0,0) rectangle (1,4);
\fill[gray] (1,1) rectangle (2,4);
\fill[gray] (2,2) rectangle (3,4);
\fill[gray] (3,3) rectangle (4,4);
\fill[black] (3,2) rectangle (4,3);
\fill[blue] (1,0) rectangle (2,1);
\end{tikzpicture}}&\scalebox{0.28}{\begin{tikzpicture}
\draw (0,0)--(1,0)--(1,1)--(2,1)--(2,2)--(3,2)--(3,3)--(4,3)--(4,4)--(0,4)--(0,0);
\draw (0,4)--(6,4)--(6,0)--(0,0);
\draw[step=1.0,black,thin] (0,0) grid (4,4);
\fill[gray] (0,0) rectangle (1,4);
\fill[gray] (1,1) rectangle (2,4);
\fill[gray] (2,2) rectangle (3,4);
\fill[gray] (3,3) rectangle (4,4);
\fill[black] (2,1) rectangle (3,2);
\fill[blue] (3,2) rectangle (4,3);
\node at (5,3.5) {\Huge even};
\end{tikzpicture}}&\scalebox{0.28}{\begin{tikzpicture}
\draw (0,0)--(1,0)--(1,1)--(2,1)--(2,2)--(3,2)--(3,3)--(4,3)--(4,4)--(0,4)--(0,0);
\draw (0,4)--(6,4)--(6,0)--(0,0);
\draw[step=1.0,black,thin] (0,0) grid (4,4);
\fill[gray] (0,0) rectangle (1,4);
\fill[gray] (1,1) rectangle (2,4);
\fill[gray] (2,2) rectangle (3,4);
\fill[gray] (3,3) rectangle (4,4);
\fill[black] (2,1) rectangle (3,2);
\fill[pattern={Hatch}, pattern color=blue] (1,0) rectangle (2,1);
\fill[blue] (2,0) rectangle (3,1);
\node at (5,3.5) {\Huge odd};
\end{tikzpicture}}&\scalebox{0.28}{\begin{tikzpicture}
\draw (0,0)--(1,0)--(1,1)--(2,1)--(2,2)--(3,2)--(3,3)--(4,3)--(4,4)--(0,4)--(0,0);
\draw (0,4)--(6,4)--(6,0)--(0,0);
\draw[step=1.0,black,thin] (0,0) grid (4,4);
\fill[gray] (0,0) rectangle (1,4);
\fill[gray] (1,1) rectangle (2,4);
\fill[gray] (2,2) rectangle (3,4);
\fill[gray] (3,3) rectangle (4,4);
\fill[black] (1,0) rectangle (2,1);
\fill[blue] (3,2) rectangle (4,3);
\end{tikzpicture}}\\
State 7&State 5&State 3&State 1&State 5\\
\end{tabular}
\end{center}

Finally, recall that in State 7, row 1 has one more eliminated cell than row 2, and row 3 has one more eliminated cell than row 4, but row 3 has at least two more eliminated cells than row 3, as illustrated below.

\begin{center}
\scalebox{0.4}{\begin{tikzpicture}
\draw (0,0)--(1,0)--(1,1)--(2,1)--(2,2)--(5,2)--(5,3)--(6,3)--(6,4)--(0,4)--(0,0);
\draw (0,4)--(8,4)--(8,0)--(0,0);
\draw[step=1.0,black,thin] (0,0) grid (6,4);
\fill[white] (3.1,0.1) rectangle (3.9,3.9);
\fill[gray] (4,3) rectangle (6,4);
\fill[gray] (4,2) rectangle (5,3);
\fill[gray] (0,3) rectangle (3,4);
\fill[gray] (0,2) rectangle (3,3);
\fill[gray] (0,1) rectangle (2,2);
\fill[gray] (0,0) rectangle (1,1);
\node at (3.5,3) {$\cdots$};
\end{tikzpicture}}
\end{center}

In general, if player 2 plays in row 1, player 1 responds in row 2 and vice versa to remain in State 7. Similarly, if player 2 plays in row 4, player 1 responds with one cell in row 3 to get to State 6 or State 7. The interesting situation is when player 2 plays in row 3, since there is more than one open cell in row 3 to choose from.  If player 2 plays anywhere but rightmost possible cell of row 3, player 1 plays in row 4 in such a way as to return to State 6 or 7.  The remaining case is when player 2 plays the final possible cell in row 3, which we consider in two cases below, depending on the parity of open cells in row 1.

\begin{center}
\begin{tabular}{cc}
\scalebox{0.4}{\begin{tikzpicture}
\draw (0,0)--(1,0)--(1,1)--(2,1)--(2,2)--(5,2)--(5,3)--(6,3)--(6,4)--(0,4)--(0,0);
\draw (0,4)--(8,4)--(8,0)--(0,0);
\draw[step=1.0,black,thin] (0,0) grid (6,4);
\fill[white] (3.1,0.1) rectangle (3.9,3.9);
\fill[gray] (4,3) rectangle (6,4);
\fill[gray] (4,2) rectangle (5,3);
\fill[gray] (0,3) rectangle (3,4);
\fill[gray] (0,2) rectangle (3,3);
\fill[gray] (0,1) rectangle (2,2);
\fill[gray] (0,0) rectangle (1,1);
\fill[pattern={Hatch}, pattern color=black] (2,1) rectangle (3,2);
\fill[black] (4,1) rectangle (5,2);
\node at (7,3.5) {\Huge odd};
\node at (3.5,3) {$\cdots$};
\end{tikzpicture}}&\scalebox{0.4}{\begin{tikzpicture}
\draw (0,0)--(1,0)--(1,1)--(2,1)--(2,2)--(5,2)--(5,3)--(6,3)--(6,4)--(0,4)--(0,0);
\draw (0,4)--(8,4)--(8,0)--(0,0);
\draw[step=1.0,black,thin] (0,0) grid (6,4);
\fill[white] (3.1,0.1) rectangle (3.9,3.9);
\fill[gray] (4,3) rectangle (6,4);
\fill[gray] (4,2) rectangle (5,3);
\fill[gray] (0,3) rectangle (3,4);
\fill[gray] (0,2) rectangle (3,3);
\fill[gray] (0,1) rectangle (2,2);
\fill[gray] (0,0) rectangle (1,1);
\fill[pattern={Hatch}, pattern color=black] (2,1) rectangle (3,2);
\fill[black] (4,1) rectangle (5,2);
\node at (7,3.5) {\Huge even};
\node at (3.5,3) {$\cdots$};
\end{tikzpicture}}\\
\scalebox{0.4}{\begin{tikzpicture}
\draw (0,0)--(1,0)--(1,1)--(2,1)--(2,2)--(5,2)--(5,3)--(6,3)--(6,4)--(0,4)--(0,0);
\draw (0,4)--(8,4)--(8,0)--(0,0);
\draw[step=1.0,black,thin] (0,0) grid (6,4);
\fill[white] (3.1,0.1) rectangle (3.9,3.9);
\fill[gray] (0,3) rectangle (6,4);
\fill[gray] (0,2) rectangle (5,3);
\fill[gray] (0,1) rectangle (2,2);
\fill[gray] (0,0) rectangle (1,1);
\fill[pattern={Hatch}, pattern color=black] (2,1) rectangle (4,2);
\fill[black] (4,1) rectangle (5,2);
\fill[pattern={Hatch}, pattern color=blue] (1,0) rectangle (4,1);
\fill[blue] (4,0) rectangle (5,1);
\node at (7,3.5) {\Huge odd};
\node at (3.5,3) {$\cdots$};
\end{tikzpicture}}&\scalebox{0.4}{\begin{tikzpicture}
\draw (0,0)--(1,0)--(1,1)--(2,1)--(2,2)--(5,2)--(5,3)--(6,3)--(6,4)--(0,4)--(0,0);
\draw (0,4)--(8,4)--(8,0)--(0,0);
\draw[step=1.0,black,thin] (0,0) grid (6,4);
\fill[white] (3.1,0.1) rectangle (3.9,3.9);
\fill[gray] (0,3) rectangle (6,4);
\fill[gray] (0,2) rectangle (5,3);
\fill[gray] (0,1) rectangle (2,2);
\fill[gray] (0,0) rectangle (1,1);
\fill[pattern={Hatch}, pattern color=black] (2,1) rectangle (3,2);
\fill[black] (4,1) rectangle (5,2);
\fill[pattern={Hatch}, pattern color=blue] (1,0) rectangle (3,1);
\fill[blue] (3,0) rectangle (4,1);
\node at (7,3.5) {\Huge even};
\node at (3.5,3) {$\cdots$};
\end{tikzpicture}}\\
State 1&State 4\\
\end{tabular}
\end{center}
\end{proof}

\begin{lemma}\label{L:end}
If the board is in a state from $\widehat{S}$ with non-zero open cells in row 1, no matter what move player 2 makes, player 1 has a response to turn the board to a state in $\widehat{E}$.
\end{lemma}

\begin{proof}
First notice that if we consider a move by player 2 followed by the prescribed response of player 1 in Lemma \ref{L:midgame}, at most two cells from row 1 are eliminated in that pair of turns.  In fact, the only time two cells from row 1 are eliminated in the same pair of turns is starting from State 2.  So, it is sufficient to follow the response prescribed in Lemma \ref{L:midgame} until there is one open cell in row 1 for States 1, 5, 6, and 7, or until there are two open cells in row 1 for States 2, 3, and 4, and then consider how one may adapt strategies to obtain a state in $\widehat{E}$ at that point of game play.

In State 1, if there is only one open cell in row 1, this is already a state in $\widehat{E}$.

From State 2, if there are two open cells in row 1, and player 2 plays in row 1 or row 4, player 1 responds as in Lemma \ref{L:midgame} to get to State 1 with one open cell in row 1, which is a state in $\widehat{E}$.  If player 2 plays in row 2, player 1 plays the last cell in row 2 to get to a state in $\widehat{E}$.  If player 2 plays in row 3, player 1 responds in row 1 to get to a state in $\widehat{E}$.

From state 3, if there are two open cells in row 1, and player 2 plays in row 1, row 2, or row 4, player 1 responds as in Lemma \ref{L:midgame} to get to State 5 or State 6 with one open cell in row 1, both of which we consider below.  If player 2 plays in row 3, player 1 responds as in Lemma \ref{L:midgame} to get to State 2, still with two open cells in row 1, which was considered above.

From State 4, if there are two open cells in row 1, following game play as in Lemma \ref{L:midgame} leads to State 6 with one open cell in row 1, considered below, or State 2 with two open cells in row 1, considered above.

From State 5, if there is one open cell in row 1, and player 2 takes the last cell in row 1, then player 1 takes penultimate cell in row 3 and vice versa to get to a state in $\widehat{E}$.  If player 2 takes the last cell in row 2, then player 1 takes the first cell in row 3 and vice versa to get to a state in $\widehat{E}$.  If player 2 moves in row 4, player 1 takes the last cell in row 1 to get to a state in $\widehat{E}$.

From State 6 when there is one open cell in row 1, if player 2 moves in row 1 or row 2, player 1 takes the final cell of row 2 to get to a state in $\widehat{E}$.  If player 2 moves in row 3 or row 4, player 1 takes cell $(a-3,4)$ to get to a state in $\widehat{E}$.

From State 7 when there is one open cell in row 1, if player 2 moves in row 1 or row 2, player 1 takes the final cell of row 2 to get to a state in $\widehat{E}$.    If player 2 takes the cell $(a-3,3)$, then player 1 takes cell $(a-3,4)$ to get to a state in $\widehat{E}$.  Otherwise if player 2 takes a different cell $(c,3)$ with $c<a-3$, then player 1 responds by taking cell $(c-1,4)$, and if player 2 plays in row 4, player 1 takes one cell in row 3, returning to State 7.
\end{proof}

\begin{proof}[of Theorem \ref{T:fivewin}]
By Lemma \ref{L:start}, player 1 has a sequence of moves leading to a state in $\widehat{S}$.

By Lemma \ref{L:midgame}, if the board is in a state in $\widehat{S}$ at the end of player 1's turn, no matter what move player 2 makes, player 1 has a response to return the board to a state in $\widehat{S}$.

By Lemma \ref{L:end}, if the board is in a state in $\widehat{S}$ with at most two open cells remaining in row 1, no matter what move player 2 makes, player 1 has a response to turn the board to a state in $\widehat{E}$.

Once the board is in a state from $\widehat{E}$, player 1 has a tit-for-tat response to any move from player 2.  In particular, there are either two rows (or two columns, or one row and one column) with open cells.  Without loss of generality consider the case with two rows of open cells.  When player 2 takes cell $(a-1,3)$ or cell $(a-2,4)$, player 1 takes cell $(a-1,4)$ and wins the game.  Until then, when player 2 takes cells from one row, player 1 takes the same number of cells from the other row.
\end{proof}

While we provided a strategy for player 1 to win, we admit that the proof of this strategy is a long list of case work. This strategy was determined by having a computer search through all $\binom{(a-1)+4}{4}$ possible shadings of an $(a-1) \times 4$ board and recursively label each as a winning or losing position for player 1, and then combing through winning states by hand to describe patterns within the of winning states that could be used as $\widehat{S}$.  This is certainly not the only winning strategy for player 1 or the only set that could be used for $\widehat{S}$ in a similar strategy.  This strategy is also notably more complex than the arguments for $b<5$, and while we conjecture a strategy for a player 1 win exists for larger $b$, the current methodology becomes increasingly cumbersome. 

The set of states $\widehat{S}$ contains states where moves are available in any of the four rows.  These are convenient descriptions of families of board shadings, where the lattice path dividing the eliminated cells from the open cells is translated horizontally, depending on $k$.  However, these kinds of states are only convenient descriptions of families of states for player 1 to aim for once cells have been eliminated in all but the final row, so that moves are available in all four rows, rather than restricted to the top 2 or 3 rows of the board.  This means that to come up with a similar strategy for larger $b$, the starting strategy of the game will take longer, to give enough moves for each of the top $b-2$ rows of the board to have non-zero eliminated cells.

\section{Other Directions}\label{S:other}

We conclude with comments on several other directions of interest.

\subsection{Strategy for normal play}

In most of this paper, we considered a permutation game where the first player to complete either an $I_a$ or a $J_b$ pattern loses, i.e., the mis\`{e}re version of the game. Here we consider normal play, i.e., the first player to complete either an $I_a$ pattern or a $J_b$ pattern \emph{wins}.

We describe the strategy in terms of the same boards as before.  In both games, the game is played on an $(b-1) \times (a-1)$ board.  In mis\`{e}re play, the player who claims the lower right corner, indexed as $(a-1,b-1)$, is therefore the winner, since the entire board is eliminated, and their opponent must either complete an $I_a$ or $J_b$ pattern as the next move.

However, in normal play, the penultimate move will be in column $a-1$ or row $b-1$, giving the player an opportunity to finish the appropriate monotone subsequence.  Since playing in column $a-1$ or row $b-1$ gives the opponent a win, players seek to \emph{not} play in this row and column.  In other words, players wanting to win, try to restrict themselves to the $(b-2) \times (a-2)$ subboard.  A player who claims cell $(a-2,b-2)$ then forces their opponent to play in the last row or last column, which allows the player who claimed $(a-2,b-2)$ to win.  In other words, normal play requires the same strategy as mis\`{e}re play, but on a board with one fewer row and one fewer column.  This means that player 1 has a winning strategy for normal play when $4 \leq b \leq 6$.  The parity of $a$ determines the winner when $b=3$. And player 2 is guaranteed a win for normal play when $b=2$ merely by completing a decreasing $J_2$ pattern on their first turn.  This observation exactly matches Theorem 9 of \cite{AAAHHMS09}.

More generally, Section 6 of \cite{AAAHHMS09} conjectures that normal play results in a first player win for $a \geq b \geq 4$.  The analysis of this paper supports that conjecture by addressing the $b=6$ case directly.  It remains to give a specific strategy for $b \geq 7$.

\subsection{Permutations and boards}

As $b$ increases, it is convenient to phrase the strategies of this paper in terms of board shading rather than in terms of the actual underlying permutations in the original problem statement.  However, at any point in the game, we may ask ``how many permutations would have produced this particular board shading?''  Clearly when $b=2$ and $b=3$, the permutations are unique.  But for larger boards, multiple permutations would result eliminating the same portion of the board.

\subsection{Number of moves}

Table \ref{T:range} shows the maximum and minimum number of moves possible in playing the Erd\H{o}s-Szekeres game for $a \geq b$ and $2 \leq b \leq 5$.  The minimum is given by two players forming a decreasing permutation of length $b$.  The maximum is given by the $(a-1)(b-1)+1$ bound given by Theorem \ref{T:ES}.  However the number of moves actually used by the strategies in this paper are in general somewhere between the two.  When $b=2$, the two players create an $I_a$ permutation.  When $b=3$, player 1 makes a strategic move to take cell $(a-1,2)$ which results in a final permutation one shorter than the maximum length in Theorem \ref{T:ES}.  When $b=4$, in general when player 2 makes a move that eliminates 2 cells, player 1 responds by eliminating 1 cell and vice versa.  The difference between $2a-1$ or $2a-3$ total moves is decided by how player 2 handles the endgame.  When $b=5$, although player 1 has a clear response to any move made by player 2, player 2 has far more options for what to do that have an impact on how quickly the game passes.  At its slowest, players take turns eliminating one cell at a time, and any turn that eliminates multiple cells comes from the starting moves or the end game.  At its fastest, player 2's turn and player 1's response eliminate 4 cells, resulting in a permutation half as long as the maximum, minus a small finite number $C$ of extra cells eliminated by choices in the start and end game.

\begin{table}
\begin{center}
\begin{tabular}{|c|c|c|c|}
\hline
$b$&minimum moves & maximum moves &actual moves\\
&&&using strategy\\
\hline
2&2&$a$&$a$\\
\hline
3&3&$2a-1$&$2a-2$\\
\hline
4&4&$3a-2$& $2a-3$ or $2a-1$\\
\hline
5&5&$4a-3$&between $2a-C$ and $4a-6$\\
\hline
\end{tabular}
\end{center}
\caption{Range of moves used in playing an $(a,b)$-permutation game.}
\label{T:range}
\end{table}

A number of interesting follow up questions remain.  For $b=5$, is there a more efficient winning strategy than the one presented here?  For larger $b$, what strategies exist, and how do they compare proportionally to the maximum length game guaranteed by Theorem \ref{T:ES}?

\subsection{Number of winning positions}

Considering the board shading interpretation of the game results in additional interesting questions.  We can shade a legal region of an $(a-1) \times (b-1)$ board in $\binom{(a-1)+(b-1)}{a-1}$ ways, by choosing the lattice path that divides the eliminated cells from the open cells.  Table \ref{T:percent} gives data for various choices of $a$ and $b$.  The first number is the number of game positions in class $\mathcal{P}$, while the number in parentheses is this number divided by $\binom{(a-1)+(b-1)}{a-1}$; in other words, the second number is the percent of game positions are in $\mathcal{P}$ (and thus are candidates to help form a set analogous to $\widehat{S}$ in the strategy for $a \geq b=5$).  Of note, when $b$ is constant but $a$ increases, these percentages overall decrease.  However, they show differences in parity, which is reflected in how the states of $\widehat{S}$ relied on the parity of open cells in row 1.  Do these values converge on a non-zero value as $a$ increases?  If so, what is it?

\begin{table}
\begin{center}
\scalebox{0.8}{
\begin{tabular}{|c||c|c|c|c|c|c|c|c|}
\hline
$b \backslash a$&2&3&4&5&6&7&8&9\\
\hline
\hline
2&1 (0.5)&2 (0.67)&2 (0.5)&3 (0.6)&3 (0.5)&4 (0.57)&4 (0.5)&5 (0.56)\\
\hline
3&&2 (0.33)&3 (0.3)&4 (0.27)&5 (0.24) &6 (0.21)&7 (0.19)&8 (0.18)\\
\hline
4&&&6 (0.3) &10 (0.29)&15 (0.27)&21 (0.25)&28 (0.23)&36 (0.22)\\
\hline
5&&&&18 (0.26)&31 (0.25)&46 (0.22)&67 (0.2)&91 (0.18)\\
\hline
6&&&&&58 (0.23)&103 (0.22)&164 (0.21)&253 (0.2)\\
\hline
\end{tabular}}

\end{center}
\caption{Total number (and percentage) of game positions on an $(a-1) \times (b-1)$ board in class $\mathcal{P}$}
\label{T:percent}
\end{table}

\acknowledgements
\label{sec:ack}

The author is grateful to two anonymous referees who provided detailed and helpful feedback that improved the presentation of this paper.



\end{document}